\documentclass{amsart}
\usepackage{amsfonts,amssymb,amsthm,eucal,amsmath,bbold,verbatim}
\usepackage[usenames, dvipsnames]{color}
\usepackage{graphicx}
\usepackage{xy}
\input xy
\xyoption{all}
\allowdisplaybreaks

\newtheorem{thm}{Theorem}[section]
\newtheorem{theorem}[thm]{Theorem}
\newtheorem{lemma}[thm]{Lemma}

\newtheorem{cor}[thm]{Corollary}

\theoremstyle{definition}
\newtheorem{definition}[thm]{Definition}

\newtheorem{claim}[thm]{Claim}

\newcommand{\R}{\mathbb{R}}

\newcommand{\E}{\mathbb{E}}
\newcommand{\N}{\mathbb{N}}

\newcommand{\vol}{\mathop{\mathrm{vol}}}

\newcommand{\ds}{\displaystyle}

\newcommand{\X}{\mathcal{X}}
\newcommand{\Y}{\mathcal{Y}}
\renewcommand{\P}{\mathbb{P}}

\newcommand{\var}{\mathrm{Var}}
\newcommand{\cov}{\mathrm{Cov}}

\newcommand{\s}{\widetilde{S}}
\newcommand{\p}{\mathcal{P}}
\renewcommand{\1}{\mathbb{1}}
\newcommand{\h}{\tilde{h}}

\begin{document}

\title[Limit theorems for Betti numbers]{Limit theorems for Betti
  numbers of \\random simplicial complexes} \author{Matthew Kahle}
\thanks{M.\ Kahle's research was supported by the NSF Research
  Training Group grant in geometry and topology of Stanford
  University.}  \address{School of Mathematics, Institute for Advanced
  Study, Einstein Drive, Princeton NJ 08540, U.S.A.}
\email{mkahle@math.stanford.edu} \author{Elizabeth Meckes}
\thanks{E.\ Meckes's research was supported by an American Institute of
  Mathematics Five-year Fellowship and NSF grant DMS-0852898.}
\address{Case Western Reserve University, Cleveland, OH 44106, U.S.A.}
\email{elizabeth.meckes@case.edu} \date{\today}

\maketitle

\begin{abstract}
There have been several recent articles studying homology of various
types of random simplicial complexes.  Several theorems have concerned
thresholds for vanishing of homology, and in some cases expectations
of the Betti numbers.  However little seems known so far about
limiting distributions of random Betti numbers.
   
In this article we establish Poisson and normal approximation theorems for Betti numbers of different kinds of random simplicial complex: Erd\H{o}s-R\'enyi random clique complexes, random Vietoris-Rips complexes, and random \v{C}ech complexes.  These results may be of practical interest in topological data analysis.

\end{abstract}

\section{Introduction}

Several papers have recently appeared concerning the topology of
random simplicial complexes \cite{clique, bhk, neighborhood, Linial,
Meshulam, triangulated, geometric}.  The results so far identify
thresholds for vanishing of homology, or compute the
expectation of the Betti numbers $\E[\beta_k]$ (i.e. the expected rank
of these groups).  In this article we prove Poisson and normal
approximation theorems for $\beta_k$ for three models of random
simplicial complex. The complexes themselves are defined precisely
and given further motivation in the following sections but we first outline our
results.
 
The first model considered is that of the Erd\H{o}s-R\'enyi random clique complex $X(n,p)$,
a higher dimensional analogue of the Erd\H{o}s-R\'enyi random graph
$G(n,p)$.  It was shown in \cite{clique} that for each $k$ and a
certain range of $p=p(n)$, $\beta_k \neq 0$ asymptotically almost surely (a.a.s\.), and in this regime, a formula for the
asymptotic size of $\E[\beta_k]$in terms of $p$ is given.
(Outside of this regime it is conjectured that $\beta_k =0 $ a.a.s.\, and some evidence for the conjecture is given in \cite{clique}.)
Here we prove
 a Central Limit Theorem for $\beta_k$. That is, we show that $$ \frac{ \beta_k -
    \E[\beta_k] }{ \sqrt{\var [ \beta_k]}} \Rightarrow
  \mathcal{N}(0,1),$$ as $n \to \infty$, where $\mathcal{N}(0,1)$ is
  the normal distribution with mean $0$ and variance $1$.

\begin{figure}\label{ER-fig}
\begin{centering}
\includegraphics{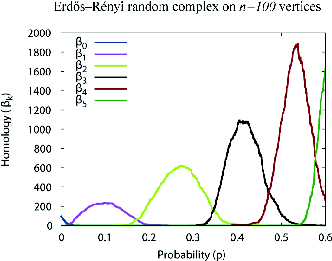}
\end{centering}
\caption{The Betti numbers of $X(n,p)$ plotted vertically against edge
  probability $p$; in this example $n=100$. \emph{Computation and
    graphic courtesy of Afra Zomorodian.}}
\label{fig:gnp}
\end{figure}

The second model considered is the random \v{C}ech complex.
This model is a higher-dimensional analog of the random geometric graph;
the underlying graph is 
a random geometric graph and the presence of 
$(k-1)$-dimensional faces is determined by
$k$-fold intersections of balls centered about the vertices. \v{C}ech complexes are homotopy
equivalent to Edelsbrunner and M\"{u}cke's {\it alpha shapes}, widely
applied in computational geometry and topology \cite{alpha}. The
analysis needed to obtain limit theorems for the Betti numbers
of random \v{C}ech complexes 
is more subtle that what is needed for 
the Erd\"os-R\'enyi model; to prove the
normal and Poisson approximation theorems we must first establish limit
theorems for certain hypergraph counts, extending some of Mathew
Penrose's results for subgraph counts for geometric random graphs
\cite{penrose}.

The final type of complex considered
 is the random Vietoris-Rips complex, denoted $VR(n,r)$. This is similar to the random \v{C}ech complex; 
the construction is to take the clique complex of a random geometric
graph.  (A useful 
reference for geometric random graphs is
\cite{penrose}.) The topology is very different than for the clique complex of the Erd\H{o}s-R\'enyi random graph;
for the contrast between $X(n,p)$ and $VR(n,r)$ see Figures
\ref{fig:gnp} and \ref{fig:geom}.  The analysis needed to obtain limit theorems
for the Betti numbers of $VR(n,r)$ is nevertheless
essentially identical to that needed
for the random \v{C}ech complex.  A minor example of this fact is that
in both cases, since $\beta_0$ counts the number of connected components 
for the \v{C}ech and Rips complexes, $\beta_0$ is actually the same in
each of these cases and
is equal to the number of components of the random geometric graph. 
This has already been treated in
detail by Penrose \cite{penrose}, and so when convenient we will restrict 
attention to $\beta_k$ for $k\ge 1$.

The techniques throughout the paper are a combination of inequalities derived from combinatorial and topological considerations with Stein's method.  (For an introduction to topological combinatorics see \cite{Bjorner}; for a survey of 
Stein's method in proving Poisson approximation theorems see \cite{CDM}, and 
for an introduction to Stein's method for normal approximation, see 
\cite{RR}.)

\subsection{Notation and conventions}

Throughout this article, we use Bachmann-Landau big-$O$, little-$O$, and related notations. In particular, for non-negative functions $g$ and $h$, we write the following.

\begin{itemize}
\item $g(n) = O(h(n))$ means that there exists $n_0$ and $k$ such that for $n > n_0$, we have that $g(n) \le k \cdot h(n)$. (i.e.\ $g$ is asymptotically bounded above by $h$, up to a constant factor.) 
\item $g(n) = \Omega(h(n))$ means that there exists $n_0$ and $k$ such that for $n > n_0$, we have that  $g(n) \ge k \cdot h(n)$. (i.e.\ $g$ is asymptotically bounded below by $h$, up to a constant factor.) 
\item $g(n) = \Theta(h(n))$ means that $g(n) = O(h(n))$ and $g(n) = \Omega(h(n))$. (i.e.\ $g$ is asymptotically bounded above and below by $h$, up to constant factors.)
\item $g(n) = o(h(n))$ means that for every $\epsilon > 0$, there exists $n_0$ such that for $n > n_0$, we have that  $g(n) \le \epsilon \cdot h(n)$. (i.e.\ $g$ is dominated by $h$ asymptotically.)
\item $g(n) = \omega(h(n))$ means that for every $k >0$, there exists $n_0$ such that for $n > n_0$, we have that  $g(n) \ge k \cdot h(n)$. (i.e.\ $g$ dominates $h$ asymptotically.)
\end{itemize}
We may also write $A_n\simeq B_n$ if $\lim_{n\to\infty}\frac{A_n}{B_n}=1$, and 
$A_n\lesssim B_n$ if there is a constant $c$  such 
that $A_n\le c B_n$ for all $n$.

A sequence $\{X_n\}_{n=1}^\infty$ of random variables is said to {\it converge
weakly} to a limiting random variable $X$ (written $X_n\Rightarrow X$)
if $\lim_{n\to\infty}\E[f(X_n)]=\E[f(X)]$ for all bounded continuous
functions $f$ (there are several other equivalent definitions).

The {\it total variation distance} between random variables $X$ and $Y$ 
is defined by 
$$d_{TV}(X,Y):=\sup_f\big|\E[f(X)]-\E[f(Y)]\big|,$$
with the supremum taken over all continuous functions bounded by one.  Clearly,
if $d_{TV}(X_n,X)\to0$ as $n\to\infty$, then $X_n\Rightarrow X$; however, the 
topology induced by the total variation distance is stronger than the 
topology of weak convergence.

The {\it $L_1$-Wasserstein distance} or {\it Kantorovich-Rubenstein distance}
between $X$ and $Y$ is defined by 
$$d_1(X,Y):=\sup_f\big|\E[f(X)]-\E[f(Y)]\big|,$$
where the supremum is over all functions $f$ with $\sup_{x\neq y}\frac{|f(x)-f(y)|}{
|x-y|}\le 1.$  This distance also induces a topology stronger than the topology
of weak convergence.

Finally, the normal distribution with mean $\mu$ and variance $\sigma^2$ is
denoted $\mathcal{N}(\mu,\sigma^2)$, and the distribution function of the 
standard normal distribution is denoted $\Phi(t)$.

\section{Erd\H{o}s-R\'enyi random clique complexes}

Perhaps the first type of random simplicial complex studied was the
$1$-dimensional version studied by Erd\H{o}s and R\'enyi \cite{Erd1}.

\begin{definition} 
The {\it Erd\H{o}s-R\'enyi random graph} $G(n,p)$ is the probability
space of all graphs on vertex set $[n] = \{1, 2, \dots, n \}$ with
each edge included independently with probability 
$p$.
\end{definition}

The ``clique complex'' is used  to generalize $G(n,p)$ from graphs to
higher dimensional simplicial complexes.

\begin{definition} 
The {\it clique complex} $X(H)$ of a graph $H$ is a the simplicial
complex with vertex set $V(H)$ and a face for 
each set of vertices spanning a
complete subgraph of $H$.
\end{definition}

In other words, the clique complex $X(H)$ of a graph $H$ is the 
maximal simplicial complex with $1$-skeleton $H$.

This section concerns the clique complex of the Erd\H{o}s-R\'enyi  random
graph, i.e.\ $X(G(n,p))$.  For simplicity in notation, \
this is denoted $X(n,p)$.

There are several motivations for using 
 $X(n,p)$ as a model of a 
random simplicial complex.  One motivation is that $X(n,p)$ provides a
natural higher-dimensional generalization of $G(n,p)$, which has
proved extremely useful in graph theory as well as in applications.  (Other
higher-dimensional generalizations are studied in \cite{bhk, Linial, Meshulam}.)  Another
motivation comes from the fact that every simplicial complex is
homeomorphic to the clique complex of some graph (e.g. by barycentric
subdivision) \cite{Hatcher}.

One interesting feature of $X(n,p)$ is that it provides homological
analogues of the Erd\H{o}s-R\'enyi theorem, but in a {\it
  non-monotone} setting: If edges are
added at random to an empty graph, the Erd\H{o}s-R\'enyi theorem
characterizes the number of edges needed
before the graph becomes connected.  Connectivity is a monotone
graph property -- if one adds edges to a connected graph, it is still
connected.

Topologically, connectivity is equivalent to a statement about zeroth
homology $H_0(G(n,p))$ but if one asks about $H_k(X(n,p))$, $k>0$,
there is a problem -- adding edges generates higher $k$-dimensional
faces and $(k+1)$-dimensional faces at the same time.  Since
generators and relations are both being added, there is no reason that
things have to behave in a monotone way.  In fact, it 
is not just that things might not be monotone; they are non-monotone in an essential way.  In
particular, there seem to be two thresholds for higher homology -- one
where $H_k$ passes from vanishing to non-vanishing, and another where
it passes back to vanishing.

The following theorem was proved in
 \cite{clique}. For any fixed
$k>0$, let $\beta_k$ denote the dimension of $k$th homology,
i.e.\ $\beta_k = \dim [ H_k (\Delta, \mathbb{Q}) ].$

\begin{theorem}\label{exp_er} 
If $p = \omega(n^{-1/k})$ and $p=o( n^{-1/(k+1)})$ then $$ \lim_{n \to
  \infty}{\E[\beta_k (X(n,p))] \over n^k p^{k+1 \choose 2} } = {1 \over
  ( k+1)!}.$$
\end{theorem}

(In \cite{clique} explicit nontrivial homology classes are exhibited, and several
partial converses of Theorem \ref{exp_er} are proved; in particular it is shown that if $p =
O(n^{-1/k - \epsilon})$ or $p = \Omega( n^{-1/(2k+1) + \epsilon})$ for
some constant $\epsilon > 0$, then a.a.s.\ $\beta_k = 0$.)

The remainder of this section is
devoted to showing that in the
same regime, $\beta_k$ obeys a central limit theorem.

\begin{theorem} \label{clt_er} 
If $p = \omega(n^{-1/k})$ and $p=o( n^{-1/(k+1)})$ then $$ \frac{
  \beta_k (X(n,p)) - \E[\beta_k(X(n,p))] }{ \sqrt{\var [ \beta_k]}} \Rightarrow
\mathcal{N}(0,1).$$
\end{theorem}

\begin{proof}
For a finite simplicial complex $\Delta$, let $f_i(\Delta)$ (or simply
$f_i$ if context is clear) denote the number of
$i$-dimensional faces of $\Delta$.  A useful fact when proving
Theorems \ref{exp_er} and \ref{clt_er} is that 
$\beta_k$ satisfies the following ``Morse''
inequalities:
\begin{equation}\label{morse}
-f_{k-1}+f_k-f_{k+1} \le \beta_k \le f_k,
\end{equation}
 for all $k$.  These
inequalities follow from the definition of simplicial homology and the
rank-nullity law \cite{Hatcher}.

The next observation to make is that $X(n,p)$ is a clique complex, 
so $f_k$ counts the number of $(k+1)$-cliques.
Since there are $\binom{n}{k+1}$ possible $(k+1)$-cliques and each appears with
probability $p^{k+1 \choose 2}$,
$$ \lim_{n \to \infty} {\E[f_k] \over n^{k+1} p^{k+1 \choose 2}} = \frac{1}{
(k+1)!}.$$
If $p = \omega(n^{-1/k})$ then

$${\E[f_{k-1}] \over \E[f_k]}  ={ n^k p^{k \choose 2} \over n^{k+1}
  p^{k+1 \choose 2}}= {1 \over n p^k}= o(1),$$

and the same argument shows that if $p =
o(n^{-1/(k+1)})$ then $${\E[f_{k +1 }] \over \E[
    f_{k} ]} = o(1).$$

That is, in the regime of Theorems \ref{exp_er} and \ref{clt_er}, 
$$\lim_{n\to\infty}\frac{\E[f_k]}{\E[-f_{k-1}+f_k-f_{k+1}]}=1,$$
which, in light of \eqref{morse}, reproves Theorem \ref{exp_er}.

Let $\tilde{f}_k:=-f_{k-1}+f_k-f_{k+1}.$ 
The following claim together with \eqref{morse} is used 
to show that $\beta_k$ satisfies a central limit
theorem.
\begin{claim}

\

\begin{enumerate}

\item\label{var_equiv} \begin{equation*}\lim_{n\to\infty}\frac{\var(f_k)}{
\var(\tilde{f}_k)}=1.\end{equation*}
\item\label{clt_above} \begin{equation*} \frac{f_k-\E[f_k]}{\sqrt{\var(f_k)}}
\Rightarrow\mathcal{N}(0,1)\quad  {\rm as\, } n\to\infty.
\end{equation*}
\item \label{clt_below}\begin{equation*} \frac{\tilde{f}_k-\E[\tilde{f}_k]}{
\sqrt{\var(\tilde{f}_k)}}\Rightarrow\mathcal{N}(0,1)\quad
  {\rm as\, } n\to\infty.\end{equation*}
\end{enumerate}

\end{claim}

For $t\in\R$, it follows from \eqref{morse} that

$$\P\left[\frac{f_k-\E[f_k]}{\sqrt{\var(f_k)}}\le
  t\right]\le\P\left[\frac{\beta_k-\E[f_k]}{\sqrt{\var(f_k)}}\le
  t\right]\le\P\left[\frac{\tilde{f}_k-\E[f_k]}{\sqrt{\var(f_k)}}\le
  t\right].$$ 
The left-hand side tends to $\Phi(t)$ as $n\to\infty$ by part \ref{clt_above}
of the claim.  For the right-hand side, let $\epsilon>0$ and observe that
\begin{equation}\begin{split}\label{fiddling}
\P\left[\frac{\tilde{f}_k-\E[f_k]}{\sqrt{\var(f_k)}}\le
  t\right]&\le\P\left[\frac{\tilde{f}_k-\E[\tilde{f}_k]}{\sqrt{\var(\tilde{f}_k)
}}\le  t-\epsilon\right]+\P\left[\left|\frac{\tilde{f}_k-\E[\tilde{f}_k]}{
\sqrt{\var(\tilde{f}_k)}}-\frac{\tilde{f}_k-\E[f_k]}{\sqrt{\var(f_k)
}}\right|>\epsilon\right]\\&
+\P\left[\frac{\tilde{f}_k-\E[f_k]}{\sqrt{\var(f_k)}}\le
  t, \left|\frac{\tilde{f}_k-\E[\tilde{f}_k]}{\sqrt{\var(\tilde{f}_k)
}}-t\right|\le \epsilon\right].
\end{split}\end{equation}
Now, it follows from part \ref{clt_below} of the claim that the first 
term of the right-hand side of \eqref{fiddling} tends to 
$\Phi(t-\epsilon)$ and that the last is asymptotically bounded above by
$\Phi(t+\epsilon)-\Phi(t-\epsilon)$.
For the second term, first require $n$ to be large enough that
$$\left|\frac{\E[f_k]}{\sqrt{\var(f_k)}}-\frac{\E[\tilde{f}_k]}{\sqrt{\var(
\tilde{f}_k)}}\right|<\frac{\epsilon}{2}.$$
This condition together with Chebychev's inequality implies that
\begin{equation*}\begin{split}
\P\left[\left|\frac{\tilde{f}_k-\E[\tilde{f}_k]}{
\sqrt{\var(\tilde{f}_k)}}-\frac{\tilde{f}_k-\E[f_k]}{\sqrt{\var(f_k)
}}\right|>\epsilon\right]&\le\P\left[\tilde{f_k}\left|\frac{1}{\sqrt{
\var(f_k)}}-\frac{1}{\sqrt{\var(\tilde{f}_k)}}\right|>\frac{\epsilon}{2}
\right]\\&\le4\epsilon^{-2}\left(\frac{\sqrt{\var(\tilde{f}_k)}}{\sqrt{
\var(f_k)}}-1\right)^2,
\end{split}\end{equation*}  
which tends to zero for fixed $\epsilon>0$ by part \ref{var_equiv}
of the claim.  It thus follows that the right-hand side of \eqref{fiddling}
is asymptotically bounded above by $\Phi(t+\epsilon)$ as $n\to\infty$; 
as $\epsilon$ is arbitrary, this completes the proof of the central limit
theorem for $\beta_k$, modulo proof of the claim.

\medskip

To prove part \ref{var_equiv} of the claim, first write 
$$f_k=\sum_{\substack{A\subseteq\{1,\ldots,n\}\\|A|=k+1}}\xi_A,$$
where $\xi_A$ is the indicator that $A$ spans a face in $X(n,p)$; that
is, that $A$ spans a complete graph in $G(n,p)$.  Then, enumerating
pairs of subsets of size $k+1$ of $\{1,\ldots,n\}$ by the size $r$ of 
their interesection,
\begin{equation*}\begin{split}
\var(f_k)&=\sum_{A,B}\E[\xi_A\xi_B]-\left[\binom{n}{k+1}p^{\binom{k+1}{2}}
\right]^2\\&=\binom{n}{k+1}\sum_{r=0}^{k+1}\binom{k+1}{r}\binom{n-k-1}{k+1-r}
p^{2\binom{k+1}{2}-\binom{r}{2}}-\left[\binom{n}{k+1}p^{\binom{k+1}{2}}
\right]^2.
\end{split}\end{equation*}
Now, it is not
hard to see that in the range of $p$ considered here, only the $r=0,1,2$
terms contribute in the limit; there is cancellation of the terms
of order $n^{k+1}$ and $n^k$, so that the main contribution is
in fact from the $r=2$ term and 
\begin{equation}\label{f_k_var}\lim_{n\to\infty}n^{-2k}p^{(-2\binom{k+1}{2}+1)}\var(f_k)=c_k,\end{equation}
for some constant $c$ depending only on $k$.  From this it follows immediately
that 
$$\frac{\var(f_{k-1})}{\var(f_{k})}=o(1)\quad{\rm and}\quad
\frac{\var(f_{k+1})}{\var(f_{k})}=o(1),$$
for $p$ in the range specified in the statement of the theorem.

Expanding the same way as above, it is clear that 
$$\cov(f_k,f_{k+1})=\binom{n}{k+1}p^{\binom{k+1}{2}+ \binom{k+2}{2}}
\left[\sum_{r=0}^{k+1}\binom{k+1}{r}\binom{n-k-1}{k+2-r}p^{-\binom{r}{2}}-
\binom{n}{k+2}\right];$$
again there is cancellation of the terms of order $n^{k+2}$
and $n^{k+1}$ so that the leading contribution is from the $r=2$ term
and 
$$\lim_{n\to\infty}n^{-2k-1}p^{-\left(\binom{k+1}{2}+ \binom{k+2}{2}-1\right)}
\cov(f_k,f_{k+1})=c_k$$
for a (different) constant $c_k$ depending only on $k$.  Thus in the 
range of $p$ being considered,
$$\frac{\cov(f_k,f_{k+1})}{\var(f_k)}=o(1).$$
In exactly the same way, one can show that
$$\frac{\cov(f_k,f_{k-1})}{\var(f_k)}=o(1)\quad{\rm and}\quad
 \frac{\cov(f_{k-1},f_{k+1})}{\var(f_k)}=o(1),$$
completing the proof of part \ref{var_equiv} of the claim.

\smallskip

The proofs of the second and third parts both
follow from an abstract normal approximation theorem for
dissociated random variables proved (via Stein's method) in \cite{BKR}.
Part \ref{clt_above} is in fact proved there; the following is a 
a straightforward modification of their proof which obtains a central
limit theorem for the lower bound $\tilde{f}_k$.  One 
can also recover the proof of part \ref{clt_above} from what is
given below, simply by ignoring the extra terms present in $\tilde{f}_k$
beyond those coming from $f_k$.

A set $\{X_{\bf j}:{\bf j}=(j_1,\ldots,j_r)\in J\}$ for $J$ a set of $r$-tuples 
is {\it dissociated} if two subcollections of the random variables
$\{X_{\bf j}:{\bf j}\in K\}$ and $\{X_{\bf j}:{\bf j}\in L\}$ are
independent whenever $\left(\cup_{{\bf j}\in K}\{j_1,\ldots,j_r\}\right)
\cap \left(\cup_{{\bf j}\in L}\{j_1,\ldots,j_r\}\right)=\emptyset.$
Let $W:=\sum_{{\bf j}\in J}X_{\bf j},$ and for each ${\bf j}\in J$, let
$L_{\bf j}:=\{{\bf k}\in J:\{k_1,\ldots,k_r\}\cap\{j_1,\ldots,j_r\}\neq
\emptyset\}.$    That is, $L_{\bf j}$ is a dependency neighborhood for 
${\bf j}$.  If $\E X_{\bf j}=0$ and $\E W^2=1$, then it is shown in 
\cite{BKR} that
\begin{equation}\label{BKR-thm}
d_1(W,Z)\le K\sum_{{\bf j}\in J}\sum_{{\bf k},
{\bf l}\in L_{\bf j}}\Big[\E|X_{\bf j}X_{\bf k}X_{\bf l}|+\E|X_{\bf j}
X_{\bf k}|\E|X_{\bf l}|\Big],
\end{equation}
where $Z$ is a standard normal random variable.

To show that $\tilde{f}_k$ satisfies a central limit theorem, let the index
set $J$ be the potential edge sets for complete graphs on $k+e$ 
($e\in\{0,1,2\}$) vertices
in $G(n,p)$; that is, an element of $J$ is a  $\binom{k+e}{2}$-tuple of 
edges spanning a given set of $k+e$ vertices.  Each ${\bf j}\in J$
can thus be  
associated with its spanning set $A_{\bf j}$ of vertices.  If the random 
variables $X_{\bf j}$ are defined by 
$$X_{\bf j}:=\sigma^{-1}(\xi_{A_{\bf j}}-\E[\xi_{A_{\bf j}}]),$$ where 
$\sigma^2=\var(f_k)$, then $\{X_{\bf j}\}$
are evidently dissociated.

The second half of the sum from \eqref{BKR-thm}
is fairly straightforward to bound in this context.  For each ${\bf j}$,
partition $L_{\bf j}$ into the sets $L_{\bf j}^e$ of indices whose spanning
sets have size $k+e$.   
Observe that for each ${\bf j}$, if $e_j=|L_{\bf j}|-k,$ then
$$|L_{\bf j}^e|=\binom{n}{k+e}-\binom{n-k-e_j}{k+e}-(k+e_j)
\binom{n-k-e_j}{k+e-1}=
O(n^{k+e-2}).$$
Decomposing as
in the variance estimate by the size $r$ of the intersection of 
$A_{\bf j}$ and $A_{\bf k}$  and using the bound 
above for $|L_{\bf j}^f|$ yields
\begin{equation*}\begin{split}
\sum_{{\bf j}\in J}&\sum_{{\bf k}\in L_{\bf j}^{e}}\sum_{
{\bf l}\in L_{\bf j}^f}\E|X_{\bf j}X_{\bf k}|\E|X_{\bf l}|\\&\le
\sigma^{-3}
c_kn^{k+f-2}p^{\binom{k+f}{2}}\binom{n}{k+e_j}\sum_{r=2}^{k+(e_j\wedge e)}
\binom{k+e_j}{r}\binom{
n-k-e_j}{k+e-r}p^{\binom{k+e}{2}+\binom{k+e_j}{2}-\binom{r}{2}}\\&\le
\sigma^{-3}c_kn^{3k+e_j+e+f-4}p^{\binom{k+e_j}{2}+\binom{k+e}{2}+\binom{k+f}{2}-1},
\end{split}\end{equation*}
since the $r=2$ term yields the top-order contribution in the
range of $p$ considered here.  Moreover, it is easy to check that this
expression is maximized for $e_j=e=f=1$.    
Combining this estimate with \eqref{f_k_var}
shows that 
the contribution to the error from the second 
sum is bounded above by
$$\sigma^{-3}c_kn^{3k-1}p^{3\binom{k+1}{2}-1}\le \frac{c_k\sqrt{p}}{n},$$
which tends to zero as $n$ tends to infinity.

The first half of the sum is bounded similarly, although it requires
that the intersections
of three spanning sets of vertices be considered. Let
$r$ denote the number of points common to $A_{\bf j}$ and 
$A_{\bf k}$.  Let $p_1:=|A_{\bf j}\cap A_{\bf l}\cap A_{\bf k}^c|$, 
$p_2:=|A_{\bf j}\cap A_{\bf l}\cap A_{\bf k}|$ and $p_3:=|A_{\bf j}^c
\cap A_{\bf l}\cap A_{\bf k}|$.  Then
$$\E|X_{\bf j}X_{\bf k}X_{\bf l}|\le
c\sigma^{-3}p^{\binom{k+e_j}{2}+\binom{k+e_k}{2} +\binom{k+e_l}{2}-
  \binom{p_1+p_2}{2}-\binom{p_2+p_3}{2}-\binom{r}{2}+\binom{p_2}{2}},$$
where the constant $c$ simply accounts for the fact that the $X_{\bf
  j}$ have been centered.  The number of ways to choose ${\bf j}$,
${\bf k}$ and ${\bf l}$ is
\begin{align*}
\binom{n}{k+e_j}\binom{k+e_j}{r}\binom{n-k-e_j}{k+e_k-r}\binom{k+e_j-r}{p_1}\\
\times \binom{r}{p_2}\binom{k+e_k-r}{p_3}\binom{n-2k-e_j-e_k+r}{k+e_l-p_1-p_2-p_3}.
\end{align*}
Combining these two facts, it is perhaps slightly unpleasant but not too
hard to see that the main contribution to the error arises from the case that
$r=2$, $p_1+p_2=2$ (in fact only when $p_1\neq 0$), and $e_j=e_k=e_l=1$.
It follows that
\begin{equation*}\begin{split}
\sum_{{\bf j}\in J}\sum_{{\bf k},
{\bf l}\in L_{\bf j}}&\E|X_{\bf j}X_{\bf k}X_{\bf l}|\le
\sigma^{-3}c_kn^{3k-1}p^{3\binom{k+1}{2}-2}\le \frac{c_k}{n\sqrt{p}},
\end{split}\end{equation*}
which also tends to zero as $n$ tends to infinity.  This completes the
proof of  part 
\ref{clt_below} of the claim, finishing the proof of Theorem \ref{clt_er}.

\end{proof}

\section{Random \v{C}ech complexes}

The second model of random simplicial complex considered is
the random \v{C}ech complex.  This is a higher-dimensional analog of 
a geometric random graph, constructed explicitly below.  In order to
analyze this model, we use the same techniques used by Penrose \cite{penrose}
in his study of subgraph counts of random geometric graph.  The additional
spacial dependence that is inherent in the random variables we consider
presents an additional technical challenge, and means that Penrose's 
results cannot be applied directly to the problem.

Suppose that $\{X_i\}_{i=1}^\infty$ is an i.i.d.\ sequence of random
vectors in $\R^d$, with bounded density $f$.  Let
$\{r_n\}_{n=1}^\infty\subseteq\R_+$, such that $nr_n^d\xrightarrow{n
  \to\infty}0$ (the so-called ``sparse'' regime of geometric random 
graphs), and construct a
random \v{C}ech complex $\mathcal{C}(X_1,\ldots,X_n)$ on
$\{X_i\}_{i=1}^n$ as follows.  If $|X_i-X_j|\le2r_n$, put an edge
between $X_i$ and $X_j$; that is,  
the 1-skeleton of the complex is a
random geometric graph. More generally, make the convex hull of $\{X_{i_1}\ldots,X_{i_k}\}$
a face of the complex  
if the balls of 
radius $r_n$ about the points $\{X_{i_1}\ldots,X_{i_k}\}$ have non-trivial
intersection.

\begin{definition}
The points $\{x_1,\ldots,x_k\}\subseteq\R^d$ form an
{\it empty $( k-1)$-simplex} with respect to $r$ if for each
$j_o\in\{1,\ldots,k\}$, the intersection $\ds\bigcap_{\substack{1\le
    j\le k\\j\neq j_o}}B_{r}(x_j)$ is non-empty, but the intersection
$\ds\bigcap_{1\le j\le k}B_{r}(x_j)=\emptyset.$
\end{definition}
 Let
$h_r(x_1,\ldots,x_k)$ be the indicator that $\{x_1,\ldots,x_k\}$ form
an empty $( k-1)$-simplex with respect to $r$, and for a multiindex ${\bf
  i}=(i_1,\ldots,i_k)$ with $1\le i_1<\cdots<i_k\le n$, let $\xi_{\bf
  i}=h_{r_n}(X_{i_1},\ldots,X_{i_k})$. 
Let 
$$S_{n,k}:=\sum_{\substack{{\bf i}=(i_1,\ldots,i_k)\\1\le i_1<\cdots<i_k\le n}}
\xi_{\bf i};$$
that is, $S_{n,k}$ is the number of empty $(k-1)$-simplices in $
\mathcal{C}(X_1,\ldots,X_n).$  
Another object of equal importance in what follows is $\s_{n,k}$, 
the number of {\it isolated}
empty $k$-simples.  That is, if $\zeta_{(i_1,\ldots,i_k)}$ is the indicator
that $\{X_{i_1},\ldots,X_{i_k}\}$ form an empty $( k-1)$-simplex with respect to $r_n$
and that there are no edges between $\{X_j\}_{j\in\{i_1,\ldots,i_k\}}$ and 
$\{X_j\}_{j\notin\{i_1,\ldots,i_k\}}$, then 
$$\s_{n,k}=\sum_{\substack{{\bf i}=(i_1,\ldots,i_k)\\1\le i_1<\cdots<i_k\le n}}
\zeta_{\bf i}.$$

The random variables $S_{n,k}$ and $\s_{n,k}$ are related to $\beta_{k-1}$
as follows.  Firstly,
$\beta_{k-1}$ is bounded below by the number of isolated empty
$k$-simplices; that is, $\beta_{k-1}(\mathcal{C}(X_1,\ldots,X_n))\ge
\s_{n,k}.$ Furthermore, any contribution to $\beta_{k-1}$ not coming
from an isolated empty $( k-1)$-simplex comes from a component in
$\mathcal{C}(X_1,\ldots,X_n)$ on at least $k+1$ vertices.  In order
for such a component to contribute to $\beta_{k-1}$, $(k-2)$-dimensional faces. Such faces are
necessarily triangulated (by the construction of
$\mathcal{C}(X_1,\ldots,X_n)$), and so any further contribution to
$\beta_{k-1}$ contains at least one simplex on $k-1$ vertices, with
either an extra edge attached to each of two different vertices
(terminating in different places), or else an extra path of length two
attached to one vertex. 
Let $Y_{n,k}$ denote the number of
simplices in $\mathcal{C}(X_1,\ldots,X_n)$ on $k-1$ vertices with two extra
edges attached, counted once for each simplex on $k-1$ vertices which occurs 
and for each distinct pair of simplex vertices with an extra edge.  
Similarly, let $Z_{n,k}$ denote the number of
simplices in $\mathcal{C}(X_1,\ldots,X_n)$ on $k-1$ vertices with at
least one extra path of length 2 attached, counted once for each simplex which
occurs and for each vertex with a path of length two attached.  
The argument above shows that
\begin{equation}\label{bounds}
\s_{n,k}\le\beta_{k-2}(\mathcal{C}(X_1,\ldots,X_n))\le S_{n,k}+Y_{n,k}+Z_{n,k},
\end{equation}
where the trivial bound $\s_{n,k}\le S_{n,k}$ has also been used.

The limiting distribution of $\beta_{k-1}$ will follow as in the previous
section by proving the same
limit theorems for the upper and lower bounds of \eqref{bounds}.  The theorem
is the following.

\begin{thm}\label{CC_clt}

\

\begin{enumerate}
\item \label{CC_clt_zero}If $n^kr_n^{d(k-1)}\to0$ as $n\to\infty$, then 
$$\beta_k(\mathcal{C}(X_1,\ldots,X_n))\rightarrow0\quad a.a.s.\ as\ n\to\infty.$$

\item \label{CC_clt_poisson}If $n^kr_n^{d(k-1)}\to\alpha\in(0,\infty)$ as $n\to\infty$, then
$$d_{TV}(\beta_k(\mathcal{C}(X_1,\ldots,X_n)),Y)\le cnr_n^d,,$$
where $Y$ is a Poisson random variable with $\E[Y]=\E[\beta_k]$ and $c$
is a constant depending only on $d$, $k$, and $f$.

\item \label{CC_clt_normal}If $n^kr_n^{d(k-1)}\to\infty$ as $n\to\infty$ and $nr_n^d\to0$ as $n\to\infty$,
then
$$\frac{\beta(\mathcal{C}(X_1,\ldots,X_n))-\E[\beta(\mathcal{C}(X_1,\ldots,X_n))]}{
\sqrt{\var(\beta(\mathcal{C}(X_1,\ldots,X_n)))}}\Rightarrow\mathcal{N}(0,1).$$

\end{enumerate}
\end{thm}

The first step in proving Theorem
\ref{CC_clt} is to determine the order in $n$ and $r_n$ of
$\E[\s_{n,k}]$ and $\E[S_{n,k}+Y_{n,k}+Z_{n,k}]$.  In fact, slightly more is needed.   Let $A$ be an open subset of $\R^d$
such that $vol(\partial A)=0$.  Let $\X$ be a finite subset of $\R^d$,
and call $x\in\X$ the ``left-most'' point of $\X$ (denoted $LMP(\X)$)
if $x$ is the first element of $\X$ when $\X$ is ordered
lexicographically.  Now, define $S_{n,k,A}$ to be the number of empty
$( k-1)$-simplices formed from $X_1,\ldots,X_n$, such that
the left-most point of the $k$-simplex is in $A$.  Define $\s_{n,k,A}$
in the analogous way.

\begin{lemma}\label{exp-order}
For $k>1$, let $$\mu_A:=\left(\int_{A}f(x)^kdx\right)\int_{(\R^d)^{k-1}}h_1(0,y_2,\ldots,y_k)
d(y_2,\ldots,y_k).$$
Then 
$$\lim_{n\to\infty}n^{-k}r_n^{-d(k-1)}\E\left[S_{n,k,A}\right]=
\lim_{n\to\infty}n^{-k}r_n^{-d(k-1)}\E[\s_{n,k,A}]=\frac{\mu_A}{k!}.$$
\end{lemma}
Observe that $\mu_A$ depends only on $f$  and $A$ 
and can be trivially bounded by
$\|f\|_\infty^{k-1}(2^d\theta_d)^{k-1},$ where $\theta_d$ is the volume of the
unit ball in $\R^d$.

\begin{lemma}\label{exp-upper}
Let $$\mu':=\left(\int_{\R^d}f(x)^{k+1}dx\right)\int_{(\R^d)^k}
g_1^{1,2}(0,y_1,\ldots,y_k)dy_1\cdots d y_k,$$
where $g_1^{1,2}(x_0,\ldots,x_k)$ is the indicator that 
$\{x_0,\ldots,x_{k-2}\}$ form a simplex (where a
  complex is built as described on $x_0,\ldots,x_{k}$ with
  threshhold radius $1$) and that $\{x_0,x_{k-1}\}$ and
  $\{x_1,x_{k}\}$ are edges.  
Let $$\mu'':=\left(\int_{\R^d}f(x)^{k+1}dx\right)\int_{(\R^d)^k}
k_1^{1}(0,y_1,\ldots,y_k)dy_1\cdots d y_k.$$
Let $k^{1}_{1}(x_0,\ldots,x_{k})$
  be the indicator that $\{x_0,\ldots,x_{k-2}\}$ form a simplex and
  that $\{x_0,x_{k-1}\}$ and $\{x_{k-1},x_{k}\}$ are edges.
Then 
$$\lim_{n\to\infty}n^{-(k+1)}r_n^{-dk}\E[Y_{n,k}]=\frac{\mu'}{2(k-3)!},$$
and 
$$\lim_{n\to\infty}n^{-(k+1)}r_n^{-dk}\E[Z_{n,k}]=\frac{\mu''}{(k-2)!}.$$
\end{lemma}

\begin{cor}
For $S_{n,k},Y_{n,k},Z_{n,k}$ as above, 
$$\E[S_{n,k}+Y_{n,k}+Z_{n,k}]\simeq\E[\s_{n,k}].$$
\end{cor}

The proofs of these facts are identical to the proofs of the corresponsing 
facts for subgraph counts of random geometric graphs given in Chapter 3 of 
\cite{penrose}.

\medskip

This last corollary is already enough to prove part \ref{CC_clt_zero} of Theorem
\ref{CC_clt}: if $n^kr_n^{d(k-1)}\to0$ as $n\to\infty$, then 
$$\P\big[\beta_k(\mathcal{C}(X_1,\ldots,X_n)\ge 1\big]\le\E\big[
\beta_k(\mathcal{C}(X_1,\ldots,X_n)\big]\le\E\big[S_{n,k}+Y_{n,k}+
Z_{n,k}\big]\xrightarrow{n\to\infty}0.$$ 

\medskip

In order to prove part \ref{CC_clt_poisson}, the following 
abstract approximation theorem of Arratia, Goldstein, and Gordon is needed.

\begin{thm}[\cite{agg}]\label{Poi-approx}
Let $(\xi_i,i\in I)$ be a finite collection of Bernoulli random variables
with dependency graph $(I,\sim)$.  Let $p_i:=\E[\xi_i]$ and $p_{ij}:=\E[\xi_i
\xi_j].$  Let $\lambda:=\sum_{i\in I}p_i,$ and let $W:=\sum_{i\in I}\xi_i$.  
Then $$d_{TV}(W,Poi(\lambda))\le\min(3,\lambda^{-1})\left(\sum_{i\in I}\sum_{
\substack{j\sim i\\j\neq i}}p_{ij}+\sum_{i\in I}\sum_{j\sim i}p_ip_j\right). $$ 
\end{thm}

Penrose \cite{penrose} used this theorem to prove Poisson approximation
results for subgraph counts of random geometric graphs; one can follow
this approach essentially without change to prove the following
result, which holds in the entire sparse regime.

\begin{thm}\label{Poisson-bd}
With definitions as above, 
$$d_{TV}\big(S_{n,k},Poi(\E[S_{n,k}])\big)\le c_{k,d,f}\big[nr_n^d\big],$$
for a constant $c_{d,k,f}$ depending only on $d$, $k$, and $\|f\|_\infty$.
\end{thm}

\begin{cor}\label{component-Poisson}
If $n^kr_n^{d(k-1)}\to\alpha\in(0,\infty)$ as $n\to\infty$, then 
$$d_{TV}\big(\s_{n,k},Poi(\E[\s_{n,k}])\big)\le \tilde{c}_{d,k,f}\alpha(nr_n^d).$$
\end{cor}
That is, in the regime of part \ref{CC_clt_poisson} of the theorem, 
the lower bound for $\beta_k$ given in \eqref{bounds} is approximately Poisson.
 
\begin{proof}
Note that $S_{n,k}-\s_{n,k}$ is the number of empty $( k-1)$-simplices among
$\{X_,\ldots,X_n\}$ which are not isolated, and is thus bounded above
by the number of connected subsets of $\{X_,\ldots,X_n\}$ with $k+1$
points, $k$ of which form an empty $k$-simplex.  The expected
number of such sets is 
bounded by $$\binom{n}{k+1}k\|f\|_\infty^{k+1}\theta_d^{k+1}(2r_n)^{d(k-1)}(4r_n)^d
\simeq\left(\frac{k\|f\|_\infty^{k+1}\theta_d^{k+1}2^{d(k+1)}}{(k+1)!}\right)
n^{k+1}r_n^{dk},$$ so that
\begin{equation*}\begin{split}
d_{TV}(S_{n,k},\s_{n,k})&=\big|\P[S_{n,k}\in A]-\P[\s_{n,k}\in A]\big|\\&=
\big|\P[S_{n,k}\in A,S_{n,k}\neq\s_{n,k}]-\P[\s_{n,k}\in A,S_{n,k}\neq\s_{n,k}]
\big|\\&\le c_{d,k,f}n^{k+1}r_n^{dk}\\&\le\tilde{c}_{d,k,f}\alpha nr_n^d.
\end{split}\end{equation*}
Moreover, it is easy to see in general that if $Y_\alpha$ and $Y_\beta$ have Poisson
distributions with means $\alpha$ and $\beta$, respectively, then 
$d_{TV}(Y_\alpha,Y_\beta)\le |\alpha-\beta|$, and so 
$$d_{TV}(Poi(\E[S_{n,k}]),Poi(\E[\s_{n,k}]))\le c_{d,k,f}\alpha nr_n^d$$
as well.

\end{proof}

\medskip

 The
following result, proved below using Theorem \ref{Poi-approx}, 
holds throughout the sparse regime.

\begin{thm}\label{upper-Poisson}
There is a constant $c_{d,k,f}$ depending on $d$, $k$, and $f$ only,
so that with $S_{n,k},Y_{n,k},Z_{n,k}$ as above, 
$$d_{TV}(S_{n,k}+Y_{n,k}+Z_{n,k}, Poi(\E[\s_{n,k}]))\le c_{d,k,f}nr_n^d.$$
\end{thm}

The inequalities in \eqref{bounds} together with Corollary 
\ref{component-Poisson} and Theorem \ref{upper-Poisson} yield part 
\ref{CC_clt_poisson} almost immediately.

\begin{proof}[Proof of part \ref{CC_clt_poisson} of Theorem \ref{CC_clt}]
By the left-hand inequality in \eqref{bounds} and Corollary 
\ref{component-Poisson},
$$\P[\beta_{k-1}\le m]\le\P[\s_{n,k}\le m]\le \P[Y\le m]+c_{d,k,f}nr_n^d,$$
 where $Y$ is a Poisson random variable with mean $\E[\s_{n,k}]$.

By the right-hand inequality in \eqref{bounds} and Theorem 
\ref{upper-Poisson}, 
$$\P[\beta_{k-1}\le m]\ge\P[S_{n,k}+Y_{n,k}+Z_{n,k}\le m]\ge \P[Y\le m]-
c_{d,k,f}nr_n^d.$$

 As in the previous proof, $Y$ can be replaced by a Poisson random variable
with mean $\E[\beta_k(\mathcal{C}(X_1,\ldots,X_n))]$ with only a change of 
constant in the error term.
\end{proof}

\begin{proof}[Proof of Theorem \ref{upper-Poisson}]
For notational convenience, let $W_{n,k}:=S_{n,k}+Y_{n,k}+Z_{n,k}$.
  For $1\le
  p<q\le k-1$, let $g^{p,q}_{r_n}(x_1,\ldots,x_{k+1})$ be the
  indicator that $\{x_1,\ldots,x_{k-1}\}$ form a simplex (where a
  complex is built as described on $x_1,\ldots,x_{k+1}$ with
  threshhold radius $r_n$) and that $\{x_p,x_k\}$ and
  $\{x_q,x_{k+1}\}$ are edges.  Let $k^{p}_{r_n}(x_1,\ldots,x_{k+1})$
  be the indicator that $\{x_1,\ldots,x_{k-1}\}$ form a simplex and
  that $\{x_p,x_k\}$ and $\{x_k,x_{k+1}\}$ are edges.  For ${\bf j}=
  (j_1,\ldots,j_{k+1})$, let $\gamma^{p,q}_{\bf
    j}=g^{p,q}_{r_n}(X_{j_1}, \ldots,X_{j_{k+1}})$ and let
  $\eta^p_{\bf j}=k^p_{r_n}(X_{j_1},\ldots,X_{j_{k+1}}).$ Then

  \begin{align*}
W_{n,k}&=\sum_{1\le i_1<\cdots<i_k\le n}\xi_{\bf i}
   +\sum_{\substack{
    1\le j_1<\cdots<j_{k-1}\le
    n\\j_k,j_{k+1}\notin\{j_1,\ldots,j_{k-1}\}\\j_k\neq j_{k+1}}}\sum_{1\le p<q\le
k-1}\gamma^{p,q}_{\bf j}\\
& \qquad\qquad+\sum_{\substack{
  1\le j_1<\cdots<j_{k-1}\le
  n\\j_k,j_{k+1}\notin\{j_1,\ldots,j_{k-1}\}\\j_k\neq j_{k+1}}}\sum_{1\le p\le
k-1}\eta^{p}_{\bf j}.
\end{align*}

 The proof that $W_{n,k}$ has an approximate Poisson distribution
proceeds along the same lines as the proof given by Penrose for subgraph
counts.  For the
Bernoulli random variables in the sum above, one can take a dependency
graph to be ${\bf i}\sim{\bf j}$ if ${\bf i}\cap{\bf j}\neq\emptyset$.
(Abusing notation, ${\bf i}$ is also used here to denote the set of
indices from the multiindex ${\bf i}$.)  Note that it is not important
that ${\bf i}$ and ${\bf j}$ be the same size.

Now, $\E[\xi_{\bf i}]\le [(2r_n)^d\theta_d\|f\|_\infty]^{k-1}$
and if $|{\bf i}\cap{\bf i'}|=\ell$, 
then $$\E[\xi_{\bf i}\xi_{\bf i'}]\le [(2r_n)^d\theta_d\|f\|_\infty]^{2k-\ell-1},$$ 
 since if set of $k$ points forms a simplex, they must
all be in the ball of radius $2r_n$ about the first point.  
Given ${\bf i}=(i_1,\ldots,i_k)$, the number of ${\bf i'}=(i_1',\ldots,
i_k')$ with ${\bf i}\sim
{\bf i'}$ (including ${\bf i}$ itself) is 
\begin{equation*}
\binom{n}{k}-\binom{n-k}{k}=\frac{k^2n^{k-1}}{k!}+O\left(n^{k-2}\right);
\end{equation*} 
for ${\bf i}$ as above, the number of ${\bf i}=(i_1',\ldots,i_k')$ with 
$\big|{\bf i}\cap{\bf i'}\big|=\ell$ is
 \begin{equation*}
\binom{k}{\ell}\binom{n-k}{k-\ell}=\binom{k}{\ell}\frac{1}{(k-\ell)!}n^{k-\ell}+
O\left(n^{k-\ell-1}\right).
\end{equation*}
This means that the contribution to the error term (without the $\min(3,
\lambda^{-1})$ factor in front) from Theorem \ref{Poi-approx} of the form
$p_{\bf i}p_{\bf i'}$ for ${\bf i}\sim{\bf i'}$ is, to top-order in $n$,

$$\frac{kn^{2k-1}}{k!
(k-1)!}\left[(2r_n)^d\theta_d\|f\|_\infty\right]^{2k-2},$$
and the contribution from terms of the form $p_{\bf ii'}$ is (to top order)
$$\binom{n}{k}\sum_{\ell=1}^{k-1}\binom{k}{\ell}\frac{1}{(k-\ell)!}n^{k-\ell}
[(2r_n)^d\theta_d\|f\|_\infty]^{2k-\ell-1}\lesssim n^{k+1}r_n^{dk}.$$

Similar to above, $\E[\gamma^{p,q}_{\bf j}]\le
2^d\left[(2r_n)^d\theta_d\|f\|_\infty\right]^k$ and if $|{\bf j}\cap{\bf
    j'}|=\ell$, then $$\E[\gamma_{\bf j}^{p,q}\gamma_{\bf
    j'}^{p',q'}]\le
  2^{3d}\left[(2r_n)^d\theta_d\|f\|_\infty\right]^{2k+1-\ell}.$$
Given ${\bf j}=(j_1,\ldots,j_{k+1})$, the number of ${\bf j'}=
(j_1',\ldots,j_{k+1}')$ with ${\bf j}\sim{\bf j'}$ is 
$$\frac{(k+1)^2n^k}{(k+1)!}+O(n^{k-1})$$
and the number of ${\bf j'}$ with $|{\bf j}\cap{\bf j'}|=\ell$ is
$$\binom{k+1}{\ell}\frac{n^{k+1-\ell}}{(k+1-\ell)!}+O(n^{k-\ell}).$$
This yields a top-order contribution to the error from Theorem 
\ref{Poi-approx} from the $\E[\gamma_{\bf j}]\E[\gamma_{\bf j'}]$ and 
$\E[\gamma_{\bf j}\gamma_{\bf j'}]$ terms  of order
\begin{equation*}\begin{split}
\frac{(k+1)^2n^{2k+1}}{[(k+1)!]^2}&\binom{k-1}{2}^22^{2d}\left[
(2r_n)^d\theta_d\|f\|_\infty\right]^{2k}\\&+\binom{n}{k+1}\sum_{\ell=1}^{k+1}
\binom{k-1}{2}^2\binom{k+1}{\ell}\frac{n^{k+1-\ell}}{(k+1-\ell)!}
2^{3d}\left[(2r_n)^d\theta_d\|f\|_\infty\right]^{2k+1-\ell}\\
&\lesssim n^{k+1}r_n^{dk}.
\end{split}\end{equation*}

    In the same way,
$\E[\eta_{\bf j}^p]\le
2^d\left[(2r_n)^d\theta_d\|f\|_\infty\right]^k,$
and if $|{\bf j}\cap{\bf
    j'}|=\ell$, then $$\E[\eta_{\bf j}^{p}\eta_{\bf
    j'}^{p'}]\le
  2^{3d}\left[(2r_n)^d\theta_d\|f\|_\infty\right]^{2k+1-\ell},$$
thus the contribution from the terms of the form $\E[\eta_{\bf j}]\E[\eta_{
\bf j'}]$ and of the form $\E[\eta_{\bf j}\eta_{\bf j'}]$ is of the same
order as the contribution above from the corresponding $\gamma$ terms.

The cross terms are essentially the same: if $|{\bf i} \cap{\bf j}|=\ell$,
then 
\begin{equation*}\begin{split}
\E[\xi_{\bf i}\gamma_{\bf j}^{p,q}]\le2^{2d}\left[(2r_n)^d\theta_d\|f\|_\infty
\right]^{2k-\ell}\qquad&\qquad\E[\xi_{\bf i}\eta_{\bf j}^{p}]\le2^{3d}
\left[(2r_n)^d\theta_d\|f\|_\infty\right]^{2k-\ell}\\
\E[\gamma^{p,q}_{\bf i}\eta_{\bf j}^{r}]\le2^{4d}&\left[(2r_n)^d\theta_d
\|f\|_\infty\right]^{2k+1-\ell}.
\end{split}\end{equation*}

The number of ${\bf j}=(j_1,\ldots,j_{k+1})$ with ${\bf i}\sim{\bf j}$
is
$$\binom{n}{k+1}-\binom{n-k}{k+1}=\frac{n^k}{(k-1)!}+O(n^{k-1}).$$
and the number of such ${\bf j}$ with $|{\bf i}\cap{\bf j}|=\ell$ is
$$\binom{k}{\ell}\binom{n-k}{k+1-\ell}=\binom{k}{\ell}\frac{n^{k+1-\ell}}{
(k+1-\ell)!}+O(n^{k-\ell}).$$

This yields a contribution from the $\xi$-$\gamma$ cross-terms of
\begin{equation*}\begin{split}
\frac{n^{2k}}{k!(k-1)!}\binom{k-1}{2}&2^d\left[(2r_n)^d\theta_d\|f\|_\infty
\right]^{2k-1}\\&+\binom{n}{k}\sum_{\ell=0}^k\binom{k-1}{2}\binom{k}{\ell}
\frac{n^{k+1-\ell}}{(k+1-\ell)!}2^{2d}\left[(2r_n)^d\theta_d\|f\|_\infty
\right]^{2k-\ell}\\&\lesssim n^{k+1}r_n^{dk}.
\end{split}\end{equation*}
The contribution from the $\xi$-$\eta$ cross terms is the same up to
constants depending only on $k$ and $d$, and the contribution from the 
$\gamma$-$\eta$ cross terms is 
\begin{eqnarray*}
 & &\frac{(k+1)^2n^{2k+1}}{[(k+1)!]^2}(k-1)\binom{k-1}{2}2^{2d}\left[(2r_n)^d
\theta_d\|f\|_\infty\right]^{2k}\\
&+&\binom{n}{k+1}\sum_{\ell=0}^{k+1}
(k-1)\binom{k-1}{2}\binom{k+1}{\ell}\frac{n^{k+1-\ell}}{(k+1-\ell)!}
2^{4d}\left[(2r_n)^d\theta_d
\|f\|_\infty\right]^{2k+1-\ell}\\
&\lesssim & n^{k+1}r_n^{dk}.
\end{eqnarray*}


Collecting terms and using that $\lambda=
\E[W_{n,k}]\simeq n^kr_n^{d(k-1)}
\left(\frac{\mu}{k!}\right)$, Theorem \ref{Poi-approx} yields
$$d_{TV}(W,Poi(\lambda))\le c_{d,k,f}nr_n^d.$$
Again, one can replace $\lambda$ with $\E[\s_{n,k}]$ with only a loss in the 
value of the constant $c_{d,k,f}$.

\end{proof}

\bigskip

The remainder of the section is devoted to the proof of part \ref{CC_clt_normal}
of Theorem \ref{CC_clt}.   A central limit theorem for the
recentered, renormalized upper bound of $\beta_k$ given in  
\eqref{bounds} follows 
immediately from Theorem \ref{upper-Poisson} in this
range of $r_n$, by the classical
result that a Poisson random variable with mean
tending to infinity tends to a Gaussian random variable when recentered
and renormalized.
\begin{thm}\label{upper-normal}
If $nr_n^d\xrightarrow{n\to\infty}0$ and $n^kr_n^{d(k-1)}\xrightarrow\infty,$
then 
$$\frac{S_{n,k}+Y_{n,k}+Z_{n,k}-\E[\s_{n,k}]}{\sqrt{\E[\s_{n,k}]}}\Longrightarrow
\mathcal{N}(0,1)$$
as $n$ tends to infinity.
\end{thm}

Clearly the approach to the lower bound of \eqref{bounds} 
taken in the regime in which $n^kr_n^{d(k-1)}\to\alpha\in(0,\infty)$ 
also works in the case that 
$n^kr_n^{d(k-1)}$ tends to infinity but $n^{k+1}r_n^{dk}$ tends to
zero to show that $\s_{n,k}$ is approximately Gaussian in that regime as
well.  However, to deal with the regime
in which $r_n=o(n^{-1/d})$ but  $n^{k+1}r_n^{dk}$ is bounded away from zero,
a different argument is needed for the lower bound of \eqref{bounds}.  
Following Penrose, the approach taken
here is to consider the Poissonized version of the problem (the vertices
distributed as a Poisson process of intensity $nf(\cdot)$ instead of i.i.d.\ with 
density $f$), and then to recover the i.i.d.\ case.

Let $N_n$ be a Poisson random variable with mean $n$, and let $\p_n=
\{X_1,\ldots,X_{N_n}\},$ where $\{X_i\}_{i=1}^\infty$ is an i.i.d.\ sequence
of random points in $\R^d$ with density $f$.  Then $\p_n$ is a Poisson 
process with intensity $nf(\cdot)$, and one can define $S^P_{n,k}$ and 
$\s^P_{n,k}$ for the random points $\p_n$ analogously to the earlier 
definitions.  In what follows, assume that $k\ge 3$; that is, 
the empty $( k-1)$-simplices are at least empty triangles.  Empty 
1-simplices
are simply pairs of vertices which are not connected, and
different arguments are needed in that case.

In order to compute expectations for the expressions which arise in
the Poissonized case, the following results are
useful.

\begin{thm}[See \cite{penrose}]\label{one}
Let $\lambda>0$  and let $\p_\lambda$ be a Poisson process with intensity 
$\lambda f(\cdot)$.  Let 
$j\in\N$, and suppose that $h(\Y,\X)$ is a bounded measurable
function on pairs $(\Y,\X)$ with $\X$ a finite subset of $\R^d$ and 
$\Y\subseteq\X$, such that $h(\Y,\X)=0$ unless $|\Y|=j$.  Then
$$\E\left[\sum_{\Y\subseteq\p_\lambda}h(\Y,\p_\lambda)\right]=\frac{\lambda^j}{j!}
\E h(\X_j',\X_j'\cup\p_\lambda),$$
where $\X_j'$ is a set of $j$ i.i.d. points in $\R^d$ with density $f$, 
independent of $\p_\lambda$.
\end{thm}

From this, one can prove the following.
\begin{thm}\label{product}
Let $\lambda>0$ and $k,j_1,\ldots,j_k\in\N$; define $j:=\sum_{i=1}^kj_i$.  
For $1\le i\le k$, suppose
$h_i(\Y,\X)$ is a bounded measurable function of pairs $(\Y,\X)$ of
finite subsets of $\R^d$ with $\Y\subseteq\X$, such that $h_i(\Y,\X)=0$
if $|\Y|\neq j_i$.  Then
$$\E\left[\sum_{\Y_1,\subseteq\p_\lambda}\cdots\sum_{\Y_k\subseteq\p_\lambda}\left(
\prod_{i=1}^kh_i(\Y_i)\right)\1_{\{\Y_i\cap\Y_j=\emptyset\,{\rm for }\,
i\neq j\}}\right]=\E\left[\prod_{i=1}^k\left(\frac{\lambda^{j_i}}{j_i!}\right)
h_i(\X_{j_i}',\X_j'\cup\p_n)\right],$$
\end{thm}
where $\X_j'$ are $j$ i.i.d points in $\R^d$ with density $f$,  $\p_\lambda$
is a Poisson process with intensity $\lambda f(\cdot)$, and $\X_j'$ and 
$\p_\lambda$ are  independent.

\begin{proof}
Consider the case $k=2$ for simplicity (the case of larger $k$ is the same
with more notation).  Define $h(\Y,\X)$ on subsets $\Y$ of $\X$ 
of size $j_1+j_2$ by
$$h(\Y,\X):=\sum_{\substack{\Y_1\subseteq\Y\\|\Y_1|=j_1}}
h_1(\Y_1,\X)h_2(\Y\setminus\Y_1,\X).$$
Then by Theorem \ref{one},
\begin{equation*}\begin{split}
\E&\left[\sum_{\Y_1,\subseteq\p_\lambda}\sum_{\Y_2,\subseteq\p_\lambda}h_1(\Y_1,\p_n)
h_2(\Y_2,\p_n)\1_{\{\Y_1\cap\Y_2=\emptyset\}}\right]\\
&\qquad=\E\left[\sum_{\Y\subseteq\p_n}
h(\Y,\p_n)\right]\\&\qquad=\frac{\lambda^{j_1+j_2}}{(j_1+j_2)!}\E h(\X_j',\X_j'\cup
\p_n)\\&\qquad=\frac{\lambda^{j_1+j_2}}{j_1!j_2!}\E\left[h_1(\X_{j_1}',\X_j'\cup\p_n)h_2
(\X_j'\setminus\X_{j_1}',\X_j'\cup\p_n)\right].
\end{split}\end{equation*}

\end{proof}

\medskip

One can apply these results to compute the mean and variance of $\s_{n,k,A}^P$, the number of isolated
empty $k$-simplices in $\p_n$ whose left-most vertex is in the set $A$.
Recall that 
$A$ is assumed to be open
with $\vol(\partial A)=0$.  

\begin{lemma}\label{CC_Poisson_means}
For $\mu_A$ as in Lemma \ref{exp-order},
$$\lim_{n\to\infty}n^{-k}r_n^{-d(k-1)}\E\left[\s_{n,k}^P\right]=
\lim_{n\to\infty}n^{-k}r_n^{-d(k-1)}\var\left[\s_{n,k}^P\right]=\frac{\mu_A}{k!}.$$
\end{lemma}

\begin{proof}
Let $\h_{r_n,A}(\{x_1,\ldots, x_k\},\X)$ be the indicator that
$\{x_1,\ldots,x_k\}\subseteq\X$ form an isolated empty $(k-1)$-simplex in
$\X$, whose left-most point is in $A$.  Then
\begin{equation}\begin{split}\label{mean-unPoisson}
\E[\s_{n,k,A}^P]&=\E\left[\sum_{\Y\subseteq\p_\lambda}\h_{r_n,A}(\Y,\p_n)\right]
=\frac{n^k}{k!}\E\left[\h_{r_n,A}(\X_k',\X_k'\cup\p_n)\right].
\end{split}\end{equation}
Now, $\E\left[\h_{r_n,A}(\X_k',\X_k'\cup\p_n)\right]\le\E\left[h_{r_n,A}(\X_k')
\right]\simeq r_n^{d(k-1)}\mu_A$.  Note that the conditional probability that
$\X_k'$ is isolated from $\p_n$ given that $\X_k'$ forms an empty 
$(k-1)$-simplex with left-most vertex in $A$
is bounded below by the probability that there are no points of $\p_n$ in
the ball of radius $4r_n$ about $X_1$, which is given by $e^{-n\vol_f(B_{4r_n}(X_1)
)}\ge e^{-n\|f\|_\infty\theta_d(4r_n)^d},$ since $\p_n$ is a Poisson 
process with intensity $nf(\cdot)$.  It thus follows that 
$$ \E\left[\h_{r_n,A}(\X_k',\X_k'\cup\p_n)\right]\ge e^{-n\|f\|_\infty\theta_d(4r_n)^d}
\E[h_{r_n,A}(\X_k')]\simeq e^{-n\|f\|_\infty\theta_d(4r_n)^d}r_n^{d(k-1)}\mu_A.$$
Since $nr_n^d\to0$, this shows that

\begin{equation*}
\E[\s_{n,k}^P]\simeq \frac{n^kr_n^{d(k-1)}\mu_A}{k!}.\end{equation*}

 A similar approach is taken to compute the variance:
\begin{equation*}\begin{split}
\E\left[(\s_{n,k,A}^P)^2\right]&
=\E\left[\sum_{\Y\subseteq\p_n}\h_{r_n,A}(\Y,\p_n)
\right]\\&\qquad\qquad
+\E\left[\sum_{j=0}^{k-1}\sum_{\Y,\Y'\subseteq\p_n}\h_{r_n,A}(\Y,\p_n)
\h_{r_n,A}(\Y',\p_n)\1_{\{|\Y\cap\Y'|=j\}}\right].
\end{split}\end{equation*}
The first summand has already been analyzed: $\E\left[\s_{n,k,A}^P\right]
\simeq\frac{n^kr_n^{d(k-1)}\mu_A}{k!}$.  
For the second, observe first that the terms corresponding to $j\neq 0$ 
vanish:\\ $\h_{r_n,A}(\Y,\p_n)\h_{r_n,A}(\Y',\p_n)\equiv 0$ if $|\Y\cap\Y'|=j$,
because if $\Y$ and $\Y'$ both form empty $k$-simplices, then neither is
isolated.
When $j=0$, applying Theorem \ref{product} yields
\begin{equation*}\begin{split}
\E&\left[\sum_{\Y,\Y'\subseteq\p_n}\h_{r_n,A}(\Y,\p_n)\h_{r_n,A}
(\Y',\p_n)\1_{\{\Y\cap\Y'=\emptyset\}}
\right]\\&\qquad\qquad
=\frac{n^{2k}}{(k!)^2}\E\left[\h_{r_n,A}(\X_k',\X_{2k}'\cup\p_n)
\h_{r_n,A}(\X_{2k}'\setminus\X_k',\X_{2k}'\cup\p_n)\right],
\end{split}\end{equation*} 
and thus (making use of \eqref{mean-unPoisson}),
\begin{equation*}\begin{split}
\var\left[\s_{n,k,A}^P\right]=\E\left[\s_{n,k,A}^P\right]+\frac{n^{2k}}{(k!)^2}
\Big(\E&\left[\h_{r_n,A}(\X_k',\X_{2k}'\cup\p_n)
\h_{r_n,A}(\X_{2k}'\setminus\X_k',\X_{2k}'\cup\p_n)\right]\\&-\left(\E\left[
\h_{r_n,A}(\X_k',\X_k'\cup\p_n)\right]\right)^2\Big),
\end{split}\end{equation*}
Now, let $\p_n'$ be an independent copy of $\p_n$.  For notational
convenience,
denote $\X_{2k}'\setminus\X_k'$ by $\mathcal{Y}_k'$ and abbreviate
$\h_{r_n,A}$ by $\h$. Then
\begin{equation*}\begin{split}
\E&\left[\h(\X_k',\X_{2k}'\cup\p_n)
\h(\mathcal{Y}_{k}',\X_{2k}'\cup\p_n)\right]-\left(\E\left[
\h(\X_k',\X_k'\cup\p_n)\right]\right)^2\\&=
\E\left[\h(\X_k',\X_{2k}'\cup\p_n)
\h(\mathcal{Y}_{k}',\X_{2k}'\cup\p_n)-\h(\X_k',\X_k'\cup\p_n)
\h(\mathcal{Y}_{k}',\mathcal{Y}_{k}'\cup\p'_n)\right]\\&=
\E\left[\left(\h(\X_k',\X_{2k}'\cup\p_n)-\h(\X_k',\X_{k}'\cup\p_n)\right)
\h(\mathcal{Y}_{k}',\X_{2k}'\cup\p_n)\right]\\&\quad+
\E\left[\h(\X_k',\X_{k}'\cup\p_n)
\left(\h(\mathcal{Y}_{k}',\X_{2k}'\cup\p_n)-\h(\mathcal{Y}_k',\mathcal{Y}_k'\cup\p_n)
\right)\right]\\&\quad+\E\left[
\h(\X_k',\X_k'\cup\p_n)\left(\h(\mathcal{Y}_{k}',
\mathcal{Y}_{k}'\cup\p_n)- 
\h(\mathcal{Y}_{k}',\mathcal{Y}_{k}'\cup\p'_n)\right)\right]\\&=E_1+E_2+E_3.
\end{split}\end{equation*}
Now, observe that in fact $E_1=0$: the difference is non-zero if and
only if $\X_k'$ and $\mathcal{Y}_k'$ are connected by an edge, in
which case the second factor is zero.

Observe that the difference in $E_2$ is non-positive.  Furthermore, 
it is non-zero if and only if $\X_k'$ and $\mathcal{Y}_k'$ are
connected by an edge, and both $\X_k'$ and $\mathcal{Y}_k'$
form empty $k$-simplices.  This probability is bounded above by
$\|f\|_\infty^{2k-1}\theta_d^{2k-1}(2r_n)^{2d(k-1)}(8r_n)^d.$

Finally, if $\left[\cup_{i=1}^kB_{2r_n}(X'_i)\right]\cap 
\left[\cup_{i=k+1}^{2k}B_{2r_n}(X'_i)\right]=\emptyset$, then the 
two terms of $E_3$ have the same distribution by the spacial
independence property of the Poisson process.  A contribution from
$E_3$ therefore only arises if in particular  $|X_1-X_j|\le 2r_n$ for
each
$2\le j\le k$, if $|X_{k+1}-X_j|\le 2r_n$ for $k+2\le j\le 2k$, and
$|X_1-X_{k+1}|\le 8r_n$.  The probability of this event is bounded
above by $\|f\|_\infty^{2k-1}\theta_d^{2k-1}(2r_n)^{2d(k-1)}(8r_n)^d.$
It follows that
\[\var\left[\s_{n,k,A}^P\right]= \E\left[\s_{n,k,A}^P\right]+E,\]
and
\[
|E|\le\frac{n^{2k}(2r_n)^{2dk-d}}{(k!)^2}2\|f\|_\infty^{2k-1}\theta_d^{2k-1}4^d=C(f,d,k)
(nr_n^d)^k(n^kr_n^{d(k-1)}),\]
where $C(f,k,d)$ is a constant depending on $f$, $d$, and $k$.  This
completes the proof.

\end{proof}

The following abstract normal approximation theorem is another version
of the dependency graph approach to Stein's method.  It is used in what
follows to prove a central limit theorem for $\s_{n,k}^P$.

\begin{thm}[Penrose]\label{normal}
Suppose $\{\xi_i\}_{i\in I}$ is a finite collection of random variables with
dependency graph $(I,\sim)$ with maximum degree $D-1$, with $\E[\xi_i]=0$
for each $i$.  Set $W:=\sum_{i\in I}\xi_i$; suppose $\E[W^2]=1$.  Let $Z$
be a standard normal random variable.  Then for all $t\in\R$,
$$\big|\P[W\le t]-\P[Z\le t]\big|\le\frac{2}{\sqrt[4]{2\pi}}\sqrt{D^2\sum_{
i\in I}\E|\xi_i|^3}+6\sqrt{D^3\sum_{i\in I}\E|\xi_i|^4}.$$
\end{thm}

Making use of this result, we prove the following.

\begin{thm}\label{Poissonized_normal}
With notation as above, and for $n^kr_n^{d(k-1)}\to\infty$ and $nr_n^d\to0$,
$$\frac{\s_{n,k}^P-\E\left[\s_{n,k}^P\right]}{\sqrt{\var\left[\s_{n,k}^P\right]}}
\Rightarrow\mathcal{N}(0,1).$$
\end{thm}

\begin{proof}
To define a dependency graph 
for the summands of $\s_{n,k}^P$, the independence
properties of the Poisson process are exploited.  Let $\{Q_{i,n}\}_{i\in \N}$
be a partition of $\R^d$ into cubes of side length $r_n$.  For the moment, assume that $A$ is a bounded set, and let $I_A$ be the set of 
indices $i$ such that $diam(A\cap Q_{i,n})>2r_n$.  Write
 \begin{equation}\label{grid}
\s_{n,k,A}^P=\sum_{i\in I_A} \sum_{\Y\subseteq\p_n}\h_{r_n,A\cap Q_{i,n}}(\Y,\p_n).
\end{equation}
Observe that if one defines a relation $\sim$ on $I_A$ by $i\sim j$ if and
only if the Euclidean distance from $Q_{i,n}$ to $Q_{j,n}$ is 
less than
$8r_n$, then $(I_A,\sim)$ is a dependency graph for the summands in
\eqref{grid}.  The degree of vertices in this dependency graph is then
bounded by $17^d$.  

Let $\xi_i:=\sum_{\Y\subseteq\p_n}\h_{r_n,A\cap Q_{i,n}}(\Y,\p_n);$ to apply Theorem
\ref{normal}, bounds are needed for $\E|\xi_i-\E\xi_i|^p$ for $p=3,4$,
for which it suffices to have bounds on 
$\E|\xi|^p$ for $p=3,4$.  Observe that if $Z_{i}$ is the number of points 
within $2r_n$ of $Q_{i,n}$, then $Z_{i,n}$ is distributed as a Poisson 
random variable with mean $n\vol_f((Q_{i,n})_{2r_n})$, and 
$$|\xi_i|\le (Z_i)(Z_i-1)\cdots(Z_i-k+1)=:(Z_i)_k.$$
It follows that there is a constant $c$ depending only on $d$ and $f$, such
that for $\rho_n:=nr_n^d$, 
$$\E|\xi_i|^p\le\E (Z_i)_k^p\le\sum_{m=k}^\infty(m)_k^p\frac{e^{-c\rho_n}(c
\rho_n)^m}{m!}\le c'\rho_n^k$$
for some new constant $c'$ depending only on $d$, $f$, and $k$.

Note that since $A$ is bounded, 
$|I_A|$ is at worst of the order $r_n^{-d}$, with coefficient 
depending on 
$A$.  Applying Theorem \ref{normal} to $\frac{\xi_i-\E\xi_i}{\sqrt{\var(\s_{n,
k,A})}}$  gives
$$\left|\P\left[\frac{\s_{n,k,A}^P-\E \s_{n,k,A}^P}{\sqrt{\var(\s_{n,k,A}^P)}}
\le t\right]-\P[Z\le t]\right|\le c''[n^kr_n^{d(k-1)}]^{-1/4},$$
which tends to zero as $n$ tends to infinity.

To move to $A=\R_d$, let $\zeta_{n,k}(A):=\frac{\s_{n,k,A}^P
-\E[\s_{n,k,A}^P]}{\sqrt{n^kr_n^{d(k-1)}}}$ and consider $A_K:=(-K,K)^d$
and $A^K:=\R^d\setminus[-K,K]^d.$  Given $t\in\R$ and $\epsilon>0$,
\begin{equation*}\begin{split}
\P[\zeta_{n,k}(\R^d)\le t]=\P[\zeta_{n,k}&(A_K)\le t-\epsilon]-\P[\{\zeta_{n,k}(A_K)
\le t-\epsilon\}\cap\{\zeta_{n,k}(\R^d)>t\}]\\&
+\P[\{|\zeta_{n,k}(A_K)-t|<\epsilon
\}\cap\{\zeta_{n,k}(\R^d)\le t\}]\\&+\P[\{\zeta_{n,k}(A_K)\ge t+\epsilon\}\cap\{
\zeta_{n,k}(\R^d)\le t\}].
\end{split}\end{equation*}
Now, $\zeta_{n,k}(\R^d)=\zeta_{n,k}(A_K)+\zeta_{n,k}(A^K)$ almost surely since
$vol(A_K^c\cup (A^K)^c)=0$, so 
$$\big|\P[\zeta_{n,k}(\R^d)\le t]-\P[\zeta_{n,k}(A_K)\le t-\epsilon]\big|\le
\P[|\zeta_{n,k}(A^K)|\ge\epsilon]+\P[|\zeta_{n,k}(A_K)-t|<\epsilon].$$
By Chebychev's inequality and the central limit theorem already
established for bounded sets, this last expression is bounded above by 
\begin{equation*}\begin{split}
\frac{1}{\epsilon^2}\var(\zeta_{n,k}(A^K))&+
\P\left[\left|\sqrt{\frac{\var(\s_{n,k,A_K}^P)}{n^kr_n^{d(k-1)}}}Z-t\right|<
\epsilon\right]+c_K\left[(n^kr_n^{d(k-1)})^{-1/4}\right]\\&\le
\frac{1}{\epsilon^2}\var(\zeta_{n,k}(A^K))+\frac{2\epsilon\sqrt{
n^kr_n^{d(k-1)}}}{\sqrt{2\pi\var(\s_{n,k,A_K}^P)}}+c_K\left[(n^kr_n^{d(k
-1)})^{-1/4}\right]\\&\simeq \frac{1}{\epsilon^2}\frac{\mu_{A^K}}{k!}
+\frac{2\epsilon\sqrt{k!}}{\sqrt{2\pi\mu_{A_K}}}+c_K\left[(n^kr_n^{d(k
-1)})^{-1/4}\right],
\end{split}\end{equation*}
for a constant $c_K$ depending on $K$.  
Taking $n$ to infinity for $K$ and $\epsilon$ fixed yields
$$\limsup_{n\to\infty}\big|\P[\zeta_{n,k}(\R^d)\le t]-\P[\zeta_{n,k}(A_K)\le 
t-\epsilon]\big|\le \frac{1}{\epsilon^2}\frac{\mu_{A^K}}{k!}
+\frac{2\epsilon\sqrt{k!}}{\sqrt{2\pi\mu_{A_K}}},$$
which, together with  
the central limit theorem for $\zeta_{n,k}(A_K)$, implies that 
$$\limsup_{n\to\infty}\left|\P[\zeta_{n,k}(\R^d)\le t]-\P\left[
\sqrt{\frac{\var(\s_{n,k,A_K}^P)}{n^kr_n^{d(k-1)}}}Z\le 
t-\epsilon\right]\right|\le \frac{1}{\epsilon^2}\frac{\mu_{A^K}}{k!}
+\frac{2\epsilon\sqrt{k!}}{\sqrt{2\pi\mu_{A_K}}}.$$
Now, 
\begin{eqnarray*}
&&\P\left[\sqrt{\frac{\var(\s_{n,k,A_K}^P)}{n^kr_n^{d(k-1)}}}Z\le 
t-\epsilon\right]\\
&=&\Phi\left(\sqrt{\frac{n^kr_n^{d(k-1)}}{\var(\s_{n,k,A_K}^P)}}(t-
\epsilon)\right)\xrightarrow{n\to\infty}\Phi\left(\sqrt{\frac{k!}{\mu_{A_K}}}
(t-\epsilon)\right);
\end{eqnarray*}
that is, 
$$\limsup_{n\to\infty}\left|\P[\zeta_{n,k}(\R^d)\le t]-\Phi\left(\sqrt{
\frac{k!}{\mu_{A_K}}}(t-\epsilon)\right)\right|\le \frac{1}{\epsilon^2}
\frac{\mu_{A^K}}{\mu}+\frac{2\epsilon\sqrt{k!}}{\sqrt{2\pi\mu_{A_K}}}.$$
Recall that $\lim_{K\to\infty}\mu_{A_K}=\mu$ and $\lim_{K\to\infty}\mu_{A^K}=
0$.  Thus for $n$ and $K$ large enough, 
$$\left|\P[\zeta_{n,k}(\R^d)\le t]-\Phi\left(\sqrt{
\frac{k!}{\mu}}(t-\epsilon)\right)\right|\le \frac{2\epsilon\sqrt{k!}}{\sqrt{2\pi\mu}}+\epsilon.$$
Since $\Phi\left(\sqrt{
\frac{k!}{\mu}}(t-\epsilon)\right)\xrightarrow{\epsilon\to0}\Phi\left(
\sqrt{
\frac{k!}{\mu}}t\right)$ and $\epsilon$ was arbitrary, this finally
shows that
$$\lim_{n\to\infty}\left|\P[\s_{n,k}^P\le t]-\Phi\left(\sqrt{
\frac{k!}{\mu}}t\right)\right|=0.$$

\end{proof}

\medskip

The remaining work is to use this result to obtain the same 
result for $\s_{n,k}$ itself. 
To do so, the following ``de-Poissonization result'' is used.
\begin{thm}[See \cite{penrose}]\label{de-Poisson}
Suppose that for each $n\in\N$, $H_n(\X)$ is a real-valued functional on finite
sets $\X\subseteq\R^d$.  Suppose that for some $\sigma^2\ge 0$, 
\begin{enumerate}
\item $\ds\frac{1}{n}\var(H_n(\p_n))\longrightarrow\sigma^2,$
and
\item $\ds\frac{1}{\sqrt{n}}\big[H_n(\p_n)-\E H_n(\p_n)\big]
\Longrightarrow\sigma^2Z,$
for $Z$ a standard normal random variable.
\end{enumerate}
Suppose that there are constants $\alpha\in\R$ and $\gamma>\frac{1}{2}$ such
that the increments $R_{m,n}=H_n(\X_{m+1})-H_n(\X_m)$ satisfy
\begin{equation}\label{means}
\lim_{n\to\infty}\left(\sup_{n-n^\gamma\le m\le n+n^\gamma}|\E[R_{m,n}]-\alpha|
\right)=0,
\end{equation}
\begin{equation}\label{covs}
\lim_{n\to\infty}\left(\sup_{n-n^\gamma\le m<m'\le n+n^\gamma}|\E[R_{m,n}R_{m',n}]
-\alpha^2|\right)=0,
\end{equation}
and
\begin{equation}\label{vars}
\lim_{n\to\infty}\left(\frac{1}{\sqrt{n}}\sup_{n-n^\gamma\le m\le n+n^\gamma}
\E[R_{m,n}^2]\right)=0.
\end{equation}
Finally, assume that there is a constant $\beta>0$ such that, with probability
one, 
$$|H_n(\X_m)|\le\beta(n+m)^\beta.$$
Then $\alpha^2\le\sigma^2$ and as $n\to\infty$, $\frac{1}{n}\var(H_n(\X_n))
\to \sigma^2-\alpha^2$ and 
$$\frac{1}{\sqrt{n}}\big[H_n(\X_n)-\E H_n(\X_n)\big]\Longrightarrow
\sqrt{\sigma^2-\alpha^2}Z.$$

\end{thm}

In conjunction with Theorem \ref{Poissonized_normal}, this yields the following.

\begin{thm}
With notation as above, and for $n^kr_n^{d(k-1)}\to\infty$ and $nr_n^d\to0$,
$$\frac{\s_{n,k}-\E\left[\s_{n,k}\right]}{\sqrt{\var\left[\s_{n,k}\right]}}
\Rightarrow\mathcal{N}(0,1).$$
\end{thm}

\begin{proof}
Theorem \ref{de-Poisson} is applied to the functional  
$$H_n(\X):=\frac{1}{\sqrt{(nr_n^d)^{k-1}}}\sum_{\Y\subseteq\X}\h_{r_n}
(\Y,\X);$$
$\sigma^2=\frac{\mu}{k!}$ and the central limit theorem
holds for $H_n(\p_n)$ by Theorem \ref{Poissonized_normal}.  

Let $D_{m,n}:=\sum_{\Y\subseteq\X_{m+1}}\h_{r_n}(\Y,\X_{m+1})-\sum_{\Y\subseteq\X_m}
\h_{r_n}(\Y,\X_m)$, and observe that $D_{m,n}$
 is the number of isolated empty $(k-1)$-simplices in 
$\X_{m+1}$ with $X_{m+1}$ as a vertex, minus the number of empty $(k-1)$-simplices
in $\X_m$ which are isolated in $\X_m$ but connected to $X_{m+1}$.
Thus
\begin{eqnarray}\label{inc_exp}
\E[D_{m,n}]&=&\binom{m}{k-1}\E[\h_{r_n}(\X_k,\X_{m+1})]\\
&&- \binom{m}{k}\E[\h_{r_n}(\X_k,\X_m)]\P\left[X_{m+1}\in\cup_{i=1}^kB_{2r_n}(X_i)
\right]. \nonumber
\end{eqnarray}
It is clear that 
 $$(1-\|f\|_\infty\theta_d(4r_n)^d)^{m+1-k}r_n^{d(k-1)}\mu\le
\E[\h_{r_n}(\X_k,\X_{m+1})]\le r_n^{d(k-1)}\mu,$$
with the upper bound arising from removing the condition that $\X_k$ be a 
component in $\mathcal{C}(\X_{m+1})$ and the lower bound arising
by bounding below the conditional probability that $\X_k$ is a component, 
given that it forms an empty $(k-1)$-simplex.   If $\gamma<1$, then 
$\lim_{n\to\infty}(1-\|f\|_\infty\theta_d(4r_n)^d)^{m+1-k}=1$, uniformly
in $m\in[n-n^\gamma,n+n^\gamma],$
thus $\E[\h_{r_n}(\X_k,\X_{m+1})]\simeq r_n^{d(k-1)}\mu$ uniformly in 
$m\in[n-n^\gamma,n+n^\gamma]$, and the same is
true for $\E[\h_{r_n}(\X_k,\X_{m})]$.

For the second term of \eqref{inc_exp}, observe that
$$\frac{\binom{m}{k}}{\binom{m}{k-1}}\P\left[X_{m+1}\in\cup_{i=1}^kB_{2r_n}(X_i)
\right]\lesssim\frac{m}{k}\|f\|_\infty\theta_d(4r_n)^d,$$
and $\lim_{n\to\infty}mr_n^d=0$, uniformly in $m\in[n-n^\gamma,n+n^\gamma]$.  
That is, the second term is of strictly
smaller order than the first.  Thus
$$\lim_{n\to\infty}\sup_{n-n^\gamma\le m\le n+n^\gamma}\left|(nr_n^d)^{1-k}\E[D_{m,n}]
-\frac{1}{(k-1)!}\mu\right|=0.$$
This implies that
$$\lim_{n\to\infty}\sup_{n-n^\gamma\le m\le n+n^\gamma}\left|(nr_n^d)^{(1-k)/2}\E[
D_{m,n}]\right|=0,$$
since $nr_n^d\to0$ as $n\to\infty$, and so the first increment condition
of the theorem is satisfied with $\alpha=0$ and any choice of $\gamma\in
(\frac{1}{2},1)$.

Next, consider the quantity $\E[D_{m,n}D_{m',n}]$ for $m\le m'$.
Recall that 
$$D_{m,n}=\sum_{\substack{\Y\subseteq\X_m\\|\Y|=k-1}}\h_{r_n}(\Y\cup\{X_{m+1}\},
\X_{m+1})-\sum_{\substack{\Y\subseteq\X_{m}\\|\Y|=k}}\h_{r_n}(\Y,
\X_{m})\1_{\left\{X_{m+1}\in\bigcup_{y\in\Y}B_{2r_n}(y)\right\}}.$$
First consider the contribution to $\E[D_{m,n}D_{m',n}]$ from terms of the form 
$$\E\big[\h_{r_n}(\Y\cup \{X_{m+1}\},\X_{m+1})\h_{r_n}(\Y'\cup \{X_{m'+1}\},
\X_{m'+1})\big]$$ for $\Y,\Y'$ such that $\big(\Y\cup \{X_{m+1}\}\big)\cap 
\Y'=\emptyset.$
By conditioning on the event $
\h_{r_n}(\Y\cup\{X_{m+1}\},\X_{m+1})=1$, it follows that 
\begin{equation*}\begin{split}
\E[\h_{r_n}(&\Y\cup \{X_{m+1}\},\X_{m+1})\h_{r_n}(\Y'\cup\{X_{m'+1}\},\X_{m'+1})]
\simeq r_n^{2d(k-1)}\mu^2\zeta,
\end{split}\end{equation*}
where $\zeta$ is the conditional probability that $\Y'\cup X_{m'+1}$
is a component in $\X_{m'+1}$, given that it forms an empty $(k-1)$-simplex,
and that $\Y\cup X_{m+1}$ forms an empty $(k-1)$-simplex which is 
not connected to any other points of $\X_{m+1}$.  Note that
if $m=m'$ then $\zeta=0$.  Otherwise, simply bound $\zeta\le 1$, so that
these terms have asymptotic order bounded above by 
$r_n^{2d(k-1)}\mu^2$, uniformly
in $m$.  The number of such terms is bounded by $\frac{(n+n^\gamma)^{2k-2}}{
[(k-1)!]^2}.$  

Note that if
$\big(\Y\cup \{X_{m+1}\}\big)\cap \Y'\neq\emptyset,$ and $m\neq m'$,
then $\h_{r_n}(\Y\cup\{
X_{m+1}\},\X_{m+1})\h_{r_n}(\Y'\cup\{X_{m'+1}\},\X_{m'+1})\equiv 0.$  If
$m=m'$ and then it must be that $\Y=\Y'$ to get a non-zero contribution.
In this case, one gains a contribution to $\E[D_{m,n}^2]$ of 
$$\binom{m}{k-1}r_n^{d(k-1)}\mu\le\frac{(n+n^\gamma)^{k-1}r_n^{d(k-1)}\mu}{
(k-1)!}.$$

Moving on to the cross terms, if $m'=m$ then  
$$\h_{r_n}(\Y\cup\{X_{m+1}\},\X_{m+1})\h_{r_n}(\Y',\X_{m})\1_{\left\{X_{m+1}
\in\bigcup_{y\in\Y'}B_{2r_n}(y)\right\}}\equiv 0.$$
If $m<m'$ (or $m>m'$), then 
\begin{equation*}\begin{split}
\E&\left[\h_{r_n}(\Y\cup\{X_{m+1}\},\X_{m+1})\h_{r_n}(\Y',\X_{m'})\1_{\left\{X_{m'+1}
\in\bigcup_{y\in\Y'}B_{2r_n}(y)\right\}}\right]\\&\qquad\qquad
\le\E\left[\h_{r_n}(\Y\cup\{X_{
m+1}\},\X_{m+1})\h_{r_n}(\Y',\X_{m'})\right]\|f\|_\infty\theta_d(4r_n)^d.
\end{split}\end{equation*}
Again, to get a non-zero contribution, it must be that $(\Y\cup\{X_{m+1}\})
\cap\Y'=\emptyset.$  In this case, the expression above is bounded above
by 
$$(r_n^{d(k-1)}\mu)^2\|f\|_\infty\theta_d(4r_n)^d.$$
The number of such terms is bounded by $\binom{m}{k-1}\binom{m}{k}\le
\frac{(n+n^\gamma)^{2k-1}}{k!(k-1)!}.$

For the product of the second sums  from $D_{m,n}$ and $D_{m',n}$, 
we have already seen
that the conditional probability that $X_{m+1}\in\bigcup_{y\in\Y}B_{2r_n}(y)$
given $\Y$ is bounded above by $\|f\|_\infty\theta_d(4r_n)^d$, and so if
$m=m'$,
\begin{eqnarray*}
\lefteqn{\E\left[\sum_{\Y\subseteq\X_{m'}}\left(\h_{r_n}(\Y,\X_{m'})\1_{\left\{X_{m'+1}\in
\bigcup_{y\in\Y}B_{2r_n}(y)\right\}}\right)^2\right] \le} \\
& & \frac{(n+n^\gamma)^k}{k!}r_n^{d(k-1)}\mu\|f\|_\infty\theta_d(4r_n)^d,
\end{eqnarray*}
while if $\Y\neq\Y'$,
$$\left(\h_{r_n}(\Y,\X_{m'})\1_{\left\{X_{m'+1}\in
\bigcup_{y\in\Y}B_{2r_n}(y)\right\}}\right)\left(\h_{r_n}(\Y',\X_{m'})\1_{\left\{X_{m'+1}\in
\bigcup_{y\in\Y'}B_{2r_n}(y)\right\}}\right)\equiv0.$$

For $m\neq m'$,
 $\Y\subseteq\X_{m}$ and $\Y\subseteq\X_{m'}$, 
let $\xi$ be the indicator that $\Y$
forms an empty $(k-1)$-simplex and $\eta$ the indicator that it is a 
component in $\X_m$.  Let $\xi'$ and $\eta'$ be the corresponding indicators
that $\Y'$ is an empty $(k-1)$-simplex and that it is a component in $\X_{m'}$.
Let $\zeta$ and $\zeta'$ be the indicators that $X_{m+1}$ is connected to
$\Y$ and that $X_{m'+1}$ is connected to $\Y'$, respectively.  Then what
is needed is 
$$\E[\xi\eta\zeta\xi'\eta'\zeta'].$$
Note that for the product to be non-zero, it must be that 
$(\Y \cup\{X_{m+1}\})\cap\Y'=\emptyset.$
Now, 
$$\P\left[\zeta\zeta'=1\big|\xi\eta\xi'\eta'=1\right]\le\frac{\|f\|^2_\infty
\theta^2_d(4r_n)^{2d}}{\vol_f(\cap_{y\in\Y'}B_{2r_n}(y)^c)}\le
\frac{\|f\|^2_\infty
\theta^2_d(4r_n)^{2d}}{1-\|f\|_\infty\theta_d(4r_n)^d},$$
since if $\xi\eta\xi'\eta'=1$, then $\Y$ and $\Y'$ make up empty 
$(k-1)$-simplices; and morover, while nothing at all is known about $X_{m'+1}$,
it is known that $X_{m+1}$ is not connected to $\Y'$.  
Trivially, 
$\P\big[\eta\eta'=1\big|\xi\xi'=1\big]\le 1$,
and $\P[\xi\xi'=1]=\P[\xi=1]\P[\xi'=1]\simeq r_n^{2d(k-1)}\mu^2,$
since $\Y\cap\Y'=\emptyset.$
Thus
\begin{equation*}\begin{split}
\E&\left[\sum_{\Y\subseteq\X_m}\sum_{\substack{\Y\subseteq\X_{m'}\\\Y'\neq\Y}}
\h_{r_n}(\Y,\X_m)\1_{\left\{X_{m+1}\in\bigcup_{y\in\Y}B_{2r_n}(y)\right\}}
\h_{r_n}(\Y,\X_{m'})\1_{\left\{X_{m'+1}\in\bigcup_{y\in\Y'}B_{2r_n}(y)\right\}}\right]
\\&\qquad\lesssim\frac{c_{d,f}(nr_n^d)^{2k}\mu^2}{(k!)^2}.
\end{split}\end{equation*}
It now follows that $\E[D_{m,n}D_{m',n}]\lesssim c_{d,f,k}(nr_n^d)^k$ for 
all $m,m'\in[n-n^\gamma,n+n^\gamma]$ with $m\neq m'$, and 
so
$$\lim_{n\to\infty}\sup_{n-n^\gamma\le m<m'\le n+n^\gamma}(nr_n^d)^{1-k}\E[D_{m,n}D_{m',n}]
=0.$$  If $m=m'$, then $\E[D_{m,n}^2]\lesssim c_{d,f}(nr_n^d)^{k-1},$ and so
$$\lim_{n\to\infty}\sup_{n-n^\gamma\le m\le n+n^\gamma}\frac{1}{\sqrt{n}}
(nr_n^d)^{1-k}\E[D_{m,n}^2]=0.$$
Thus the increment conditions of the theorem are satisfied with $\alpha=0$.

Finally, observe that 
$$H_n(\X_m)\le\frac{\sqrt{n}m}{n^kr_n^{d(k-1)}k}\le\frac{(\sqrt{n}+m)^2}{n^kr_n^{d(k-1)}k};$$
since $n^kr_n^{d(k-1)}$ is assumed to go to infinity as $n\to\infty$, 
the polynomial boundedness condition of Theorem \ref{de-Poisson}
is satisfied
and the  central limit theorem for 
$\s_{n,k}$ is proved.

\end{proof}

\bigskip

As was previously noted, that the same central limit theorem holds for 
upper and lower bounds for $\beta_k$ given in \eqref{bounds}
immediately yields part \ref{CC_clt_normal}
of Theorem \ref{CC_clt}.

\begin{thm}\label{betti-normal}
$$\frac{\beta_{k-1}-\E[\s_{n,k}]}{\sqrt{\E[\s_{n,k}]}}\Longrightarrow
\mathfrak{N}(0,1).$$
\end{thm}

\section{Vietoris-Rips complexes}

Vietoris-Rips complexes were introduced by Leopold Vietoris in the
context of algebraic topology, and independently by Eliyahu Rips in
the context of geometric group theory.  These complexes continue to be
a useful construction in both fields, and are also useful in
computational topology -- although they do not carry the same homotopy
information that the \v{C}ech complex does, the fact that they are
determined by their underlying graph makes them much smaller in memory
and more amenable to certain kinds of calculation. 

Let $f: \mathbb{R}^d \to \mathbb{R}^{\ge 0}$ be a bounded measurable
density function and  et $\mathcal{X}_n$ denote a
set of $n$ points drawn independently from this distribution.  For any
$r>0$ define a (random geometric) graph $G(n,r)$ 
on $\mathcal{X}_n$ by inserting an edge $\{x,y\}$
whenever $d(x,y) < 2r$.  Usually 
$r=r(n)$ and we consider the limit
as $n$ tends to infinity.

The {\it random Vietoris-Rips complex} $VR (n,r)$ is the clique
complex of this random geometric graph; that is, the maximal 
simplicial complex with 1-skeleton $G(n,r)$. 
To see the contrast with $X(n,p)$, Figure \ref{VR-fig} 
has a picture of the Betti numbers of a random Rips
complex $VR(n,r)$ on $100$ uniform points in a $6$-dimensional cube,
with $n=100$ and $0 \le r \le 1$; compare with Figure \ref{ER-fig}.

\begin{figure}\label{VR-fig}
\begin{centering}
\includegraphics{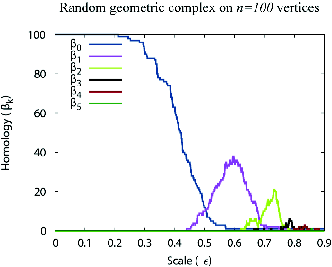}
\end{centering}
\caption{The Betti numbers of $VR(n,r)$ plotted vertically against $r$
  horizontally; $n=100$. \emph{Computation and graphic courtesy of
    Afra Zomorodian.}}
\label{fig:geom}
\end{figure}

In the sparse range of parameter, $r = o(n^{-1/d})$, a formula for the
asymptotic expectation of $\beta_k$ was given in \cite{geometric}.

\begin{theorem} \label{exp_rips} 
For $d \ge 2$, $k \ge 1$, $\epsilon>0$, and $r _n = O(n^{-1/d -
  \epsilon})$, the expectation of the $k$th Betti number $\E[\beta_k]$
of the random Vietoris-Rips complex $VR(X_n;r_n)$ satisfies $$
\frac{\E[\beta_k]}{n^{2k+2} r_n^{d(2k+1)}} \to C_k,$$ as $n \to
\infty$, where $C_k$ is a constant that depends only on $k$ and the
underlying density function $f$.
\end{theorem}

In the same regime we prove limit theorems for 
$\beta_k$.

\begin{theorem} \label{clt_rips} 
With the same hypothesis as in Theorem \ref{exp_rips}, 
\begin{enumerate}

\item if $n^{2k+2} r_n^{d(2k+1)}\to0$ as $n\to\infty$, then 
$$\beta_k(VR(X_n;r_n))\to0\qquad\qquad a.a.s.;$$

\item if $n^{2k+2} r_n^{d(2k+1)}\to\alpha\in(0,\infty)$ as $n\to\infty$,
then 
$$d_{TV}(\beta_k(VR(X_n;r_n)),Y)\le c\alpha nr_n^d,$$
where $Y$ is a Poisson random variable with $\E[Y]=\E[\beta_k]$
and $c$ is a constant depending only on $d$, $k$, and $f$;

\item if $n^{2k+2} r_n^{d(2k+1)}\to\infty$, then 
$$ \frac{ \beta_k - \E[\beta_k] }{
  \sqrt{\var [ \beta_k]}} \to \mathcal{N}(0,1).$$

\end{enumerate}

\end{theorem}

(The case $k=0$ is handled in detail by Penrose \cite{penrose}.)

The main idea of the proof of Theorem
\ref{clt_rips} is again to bound $\beta_k$
between two random variables which 
satisfy the same central limit theorem. The intuition behind the bounds is
that almost all of the homology of $VR(n,r)$ is contributed
from a single source: the octahedral components.This is essentially because they
are the smallest possible support of homology (smallest in the sense
of vertex support), in the same way that empty $(k-1)$-simplices
were the smallest possible support of homology in the previous section.

\begin{definition}
The $(k+1)$-dimensional {\it cross-polytope} is
defined to be the convex hull of the $2k+2$ points $\{ \pm e_i
\}$, where $e_1, e_2, \ldots, e_{k+1}$ are the standard basis vectors of
$\R^{k+1}$. The boundary of this polytope is a $k$-dimensional simplicial
complex, denoted $O_{k}$.
\end{definition}

Simplicial complexes which arise as clique complexes of graphs are
sometimes called {\it flag complexes}.  A useful fact in
combinatorial topology is the following; for a proof see
\cite{clique}.

\begin{lemma} \label{octa} 
If $\Delta$ is a flag complex, then any nontrivial element of $k$-dimensional
homology $H_k(\Delta)$ is supported on a subcomplex $S$ with at least
$2k+2$ vertices. Moreover, if $S$ has exactly $2k+2$ vertices, then
$S$ is isomorphic to $O_k$.
\end{lemma}

\begin{definition} 
Let $o_k(\Delta)$ (or $o_k$ if context is clear) denote the number of
induced subgraphs of $\Delta$ 
combinatorially isomorphic to the
$1$-skeleton of the cross-polytope $O_k$, and let $\tilde{o}_k(
\Delta)$ denote the 
number of components of $\Delta$ combinatorially isomorphic to the
$1$-skeleton of the cross-polytope $O_k$.
\end{definition}

\begin{definition} 
Let $f_k^{= i}(\Delta)$ denote the number of $k$-dimensional faces on
connected components containing with exactly $i$ vertices.
Similarly,
let $f_k^{\ge i}(\Delta)$ denote the number of $k$-dimensional faces
on connected components containing at least $i$ vertices.
\end{definition}

In \cite{penrose}, Penrose proved the following 
limit theorems for subgraph counts of random geometric graphs. 

\begin{thm}[Penrose]\label{subgraph-clt}
Let $\Gamma_1,\ldots,\Gamma_m$ be graphs on $v\ge2$ vertices, such that
$\P[G(v,r)\cong \Gamma_j]>0$ for each $j$.  Let $G_n(\Gamma)$ denote
the number of induced subgraphs of $G(n,r_n)$ isomorphic to $\Gamma$.
Then with $r_n$ as in the 
statement of Theorem \ref{clt_rips}, 
\begin{enumerate}
\item \label{subgraph-means}
There is a constant $\mu_j$ depending only on $\Gamma_j$ and $v$
such that $$\lim_{n\to\infty}r_n^{-d(v-1)}n^{-v}\E[G_n(\Gamma_j)]=\mu_j.$$

\item Let $Z_1,\ldots,Z_m$ be indpendent Poisson random variables with 
$\E Z_j=\E[G_n(\Gamma_j)]$.  There is a constant $c$ depending only on $m$
such that
$$d_{TV}\big[(G_n(\Gamma_1),\ldots,G_n(\Gamma_m)),(Z_1,\ldots,Z_m)\big]\le
cn^{v+1}r_n^{dv}.$$

\item Suppose that $n^vr_n^{d(v-1)}\to\infty$ as $n\to\infty$.  Let 
$\tau=\sqrt{n^vr_n^{d(v-1)}}$.  Then the joint distribution of the random
variables $\{G_n(\Gamma_j)\}_{j=1}^m$ converges to a centered Gaussian 
distribution with covariance matrix $\Sigma=diag(\mu_1,\ldots,\mu_m)$,
for $\mu_j$ as in part \ref{subgraph-means}

\end{enumerate}

\end{thm}
 
A dimension bound paired with Lemma \ref{octa} yields 
\begin{equation}\label{octo-morse}
\tilde{o}_k \le \beta_k \le \tilde{o}_k + f_k^{\ge 2k+3},
\end{equation}
in analogy to the Morse
inequalities used in the first section.

One could work with $f_k^{\ge 2k+3}$ directly, but it turns out to be sufficient 
 to overestimate $f_k^{\ge 2k+3}$ as follows.  For each 
$k$-dimensional face, consider the underlying 
$(k+1)$-clique; if it is in a component
with at least $2k+3$ vertices, extend the clique to a connected
subgraph with exactly $2k+3$ vertices and ${k+1 \choose 2} + k+2$
edges, by the following algorithm.

\begin{enumerate}
\item Set $G$ to be the $1$-skeleton of
  the complex, and initialize $H$ to be the $(k+1)$-clique.
\item Find some edge connecting $V(H)$ to $V(G) - V(H)$.  Add this
  edge (and its endpoint) to $H$.  This is always possible since by
  assumption $H$ is contained in a component with at least $2k+3$
  vertices.
\item Repeat step $2$ until $H$ has exactly $2k+3$ vertices.
\end{enumerate}

For example, let $k=2$; then  
 $$\tilde{o}_2 \le \beta_2 \le \tilde{o}_2 + f_2^{\ge
  7}.$$  Up to isomorphism, the seventeen graphs that arise when
extending a $2$-dimensional face (i.e.\ a $3$-clique) to a minimal
connected graph on $7$ vertices are exhibited in Figure
\ref{fig:betti2}.

In particular,$f_2^{\ge 7} \le \sum_{i=1}^{17} s_i,$ where
$s_i$ counts the number of subgraphs isomorphic to graph $i$ for some
indexing of the seventeen graphs in Figure \ref{fig:betti2}.

\begin{figure}
\begin{centering}
\includegraphics[width=5in]{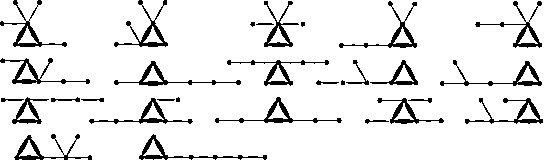}
\end{centering}
\caption{The case $k=2$: the seventeen isomorphism types of subgraphs
  which arise when extending a $3$-clique to a connected graph on $7$
  vertices with $7$ edges.  Each
  subgraph isomorphic to one of these can contribute at most $1$
  to the sum bounding the error term $f_2^{\ge 7}$.}
\label{fig:betti2}
\end{figure}

In general, 
one can express the number of graphs on $2k+3$ vertices that can 
arise from the algorithm above as a function of $k$.  Moreover, 
as is noted in \cite{penrose}, 
the number of occurances of a given graph $\Gamma$ on $v$ vertices (that is, 
the subgraph count corresponding to $\Gamma$) can be written
as a linear combination of the induced subgraph counts for those
graphs on $v$ vertices which have $\Gamma$ as a subgraph.  That is,
\begin{equation}\label{VR_morse}\tilde{o}_k\le\beta_k\le o_k+g_{2k+3},
\end{equation}
where $g_{2k+3}$ is a linear combination of the induced subgraph counts
of graphs on $2k+3$ vertices, the number of which depends only on $k$,
and the trivial bound $\tilde{o}_k\le o_k$ has been used on the right-hand side.

The
induced subgraph counts appearing on the right-hand side of \eqref{VR_morse}
are among the components of a random vector whose 
joint distribution is identified in Theorem \ref{subgraph-clt} (for two different
values of $v$), and thus limiting distributions for
$o_k$ and $g_{2k+3}$ are known in those regimes.  
Moreover, it is easy to modify Penrose's proofs (just as in the previous 
section) to show that 
$$d_{TV}(o_k+g_{2k+3},Y)\le c\alpha nr_n^d,$$
where $Y$ is a Poisson random variable with $\E[Y]=\E[o_k+g_{2k+3}]$,
which in particular yields a central limit theorem if $n^{2k+2}r_n^{d(2k+1)}\to
\infty$ as $n\to\infty$.  

To obtain the limiting distribution for the lower bound of \eqref{VR_morse}
is also just as in the previous section; all the proofs go through
in exactly the same way, and will therefore not be repeated.

For $k=1$ there are several ways of extending a $2$-clique (i.e.\ an
edge) to a connected graph on $5$ vertices and $4$ edges.  In this
case the graph must be a tree, and 
it is no longer possible
to recover the clique from the connected graph.  
However, there
are only three isomorphism types of trees on five vertices, shown in
Figure \ref{fig:betti1}.  Counting these types of subgraphs 
may therefore result in an underestimate
for $f_1^{\ge 5}$  because some edges might
get extended to the same tree.  However, each tree has only four edges, and so one can obtain the bound
 $$f_1^{\ge 5} \le 4 ( t_1 + t_2 + t_3),$$ where
$t_1, t_2, t_3$ count the number of subgraphs isomorphic to the three
trees in Figure \ref{fig:betti1}.  The proof is then the same as in the 
case $k\ge 2$.

\begin{figure}
\begin{centering}
\includegraphics[width=3.5in]{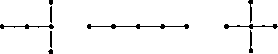}
\end{centering}
\caption{The case $k=1$: the three isomorphism types of trees on five vertices.  Each
  subgraph isomorphic to one of these can contribute at most $4$
  to the sum bounding the error term $f_1^{\ge 5}$.}
\label{fig:betti1}
\end{figure}

\section{Comments} \label{section:open}

We studied here three different kinds of random simplicial complex in order to work as generally as possible; however there are various ways in which we believe it may be possible to extend our results.\\

\paragraph 1 The random Vietoris-Rips and \v{C}ech complexes studied here are on Euclidean space, but this is mostly a matter of convenience.  It would seem that the same proofs work, mutatis mutandis, for arbitrary Riemannian manifolds.  This may be of interest in topological data analysis, as in earlier work of Niyogi, Smale, and Weinberger \cite{Smale}.

\medskip

\paragraph 2 It may be possible to extend the central limit theorems for the random Vietoris-Rips and \v{C}ech complexes into denser regimes, at least into the thermodynamic limit.  We expect, for example, that there exists some $c>0$ such that CLT's
 hold for all Betti numbers $\beta_k$ simultaneously, whenever $r \ge c n^{-1/d}$. 

\medskip

\paragraph 3 An easier  argument than those
 presented here should yield central limit theorems
  for Euler characteristic $\chi$ of geometric random complexes,  in the sparse range. Again it would be nice to know this
  this in denser regimes, and we would guess that it holds at  
  least partway into the thermodynamic limit.

\bigskip

\noindent {\bf Acknowledgements:} The authors met and began discussing
this project at the Workshop on Topological Complexity of Random Sets held
at the American Institute of Mathematics in August, 2009; many thanks to 
AIM and to the organizers of the workshop.  The authors also thank
Omer Bobrowski for pointing out a mistake in the original version
of the paper.

%


\bibliographystyle{plain}
\bibliography{empty}

\end{document}